\documentclass[a4paper,11pt]{amsart}
\usepackage{amsmath, amsfonts, amsthm}
\usepackage{amssymb,mathrsfs, verbatim}
\usepackage{graphicx,amstext}
\usepackage{amsmath}
\usepackage{mathtools}
\usepackage{amssymb}
\usepackage{latexsym}
\usepackage{ mathrsfs }
\usepackage{enumerate}
\usepackage[USenglish]{babel}

\usepackage[latin1]{inputenc}
\newcommand{\R}{\mathbb R}

\newcommand{\N}{\mathbb N}

\def\norm#1{\|#1\|}

\newcommand\blfootnote[1]{%
	\begingroup
	\renewcommand\thefootnote{}\footnote{#1}%
	\addtocounter{footnote}{-1}%
	\endgroup
}

\numberwithin{equation}{section}

\newtheorem{theorem}{Theorem}[section]
\newtheorem{proposition}{Proposition}[section]
\newtheorem{remark}{Remark}[section]
\newtheorem{lemma}{Lemma}[section]

\newtheorem{corollary}{Corollary}[section]

\newtheorem{remarks}{Remarks}[section]
\begin{document}
\title{The cubic nonlinear fractional Schr\"odinger equation on the half-line}
\author{M\'arcio Cavalcante  
	and Gerardo Huaroto }
\address{\emph{Instituto de Matem\'{a}tica, Universidade Federal de Alagoas (UFAL), Macei\'o (AL), Brazil}}
\email{marcio.melo@im.ufal.br}
\address{\emph{Instituto de Matem\'{a}tica, Universidade Federal de Alagoas (UFAL), Macei\'o (AL), Brazil}}

\email{gerardo.cardenas@im.ufal.br}

\date{}

\maketitle

\begin{abstract}
We study the cubic nonlinear
fractional  Schr\"odinger  equation with L\'evy indices $\frac{4}{3}<\alpha< 2$ posed on the half-line. More precisely,  we define the notion of a solution for this model and we obtain a result of local-well-posedness almost sharp with respect for known results on the full real line $\mathbb R$. Also, we prove for the same model
that the solution of the  nonlinear part  is smoother than the initial
data. To get our results we use the Colliander and Kenig approach based in the Riemann--Liouville fractional operator combined with Fourier restriction method of Bourgain \cite{Bourgain3} and some ideas of the recent work of Erdogan, Gurel and Tzirakis \cite{tzirakis2}. The method applies to both focusing and defocusing nonlinearities. As the consequence of our analysis we prove a smothing effect for the cubic nonlinear fractional Schr\"odinger equation posed in full line $\mathbb R$ for the case of the low regularity assumption,  which was point out at the recent work \cite{tzirakis2}.

\end{abstract}

\maketitle

\section{Introduction}
\label{INTRO}

\subsection{Presentation of the model} The one dimensional fractional cubic nonlinear Schr\"odinger equation
$$i\partial_tu(x,t)+(-\Delta)^{\alpha/2}u(x,t) =\lambda|u(x,t)|^2u(x,t),\ x,t\in\R$$
 was introduced in the
theory of the fractional quantum mechanics where the Feynmann path integrals
approach is generalized to $\alpha$-stable L\'evy process \cite{Laskin}. Also it appears in the water
wave models (for example, see \cite{Ionescu} and references therein).

In the  mathematical contex is more studied the folowing initial value problem (IVP) associated to the fractional cubic NLS equation
\begin{equation}\label{frac}
\begin{cases}
i\partial_tu(x,t)+(-\Delta)^{\alpha/2}u(x,t) =\lambda|u(x,t)|^2u(x,t),& (x,t)\in\mathbb{R}\times I,
\\[5pt]
u(x,0)=u_0(x),& x\in I,
\end{cases}
\end{equation}
where $I=\R$ or $I=\mathbb{T}$.

For $s > 1/2$, the Sobolev embedding and the energy method one
can easily show the local well-posedness in $H^s(\R)$ for $1 < \alpha < 2$. For less regular initial data, i.e. $s<\frac12$, the local well-posedness for the fractional NLS on the real line was recently studied by Cho et al.
\cite{Cho}. The authors showed that the equation is locally well-posed in $H^{s}(\R)$, for $s\geq\frac{2-\alpha}{4}$.
They also proved that the solution operator fails to be uniformly continuous in time for $s<\frac{2-\alpha}{4}$.

More recently, for the IVP \eqref{frac} Erdogan, Gurel and Tzirakis \cite{tzirakis2} proved the nonlinear part of the solution is smoother than the initial
data. More precisely, they obtained the following smoothing effect result:
\begin{theorem}\label{theorem0}
Consider the IVP \eqref{frac} on $\mathbb{R}$. Fix $\alpha \in (1,2)$. For any $s>\frac12$ and $a<2\alpha-1$, we have that the solution of \eqref{frac}  satisfies 
\begin{equation}
u(t)-e^{it(-\Delta)^{\alpha/2}}u_0\in H^{s+a}(\R)
\end{equation} and
\begin{equation*}
\begin{split}
\|u(t)-e^{it(-\Delta)^{\alpha/2}}u_0\|_{H^{s+a}(\R)}\lesssim& \|u_0\|_{H^s(\R)}\left(1+\|u_0\|_{H^{1/2^+}(\R)}^2\right)\\
&\quad +\|u(t)\|_{H^s(\R)}\|u(t)\|_{H^{1/2^+}(\R)}^2\\
&\quad +\int_0^t \|u(t')\|_{H^s(\R)}\left(\|u(t')\|_{H^{1/2+}(\R)}^2+\|u(t')\|_{H^{1/2+}(\R)}^4\right)dt'
\end{split}
\end{equation*}
for all $t$ in the maximal interval of existence. 

In particular, for $s\geq \alpha$ and $a<2\alpha-1$ we have
\begin{equation}
\|u(t)-e^{it(-\Delta)^{\alpha/2}}u_0\|_{H^{s+a}(\R)}\lesssim \|u_0\|_{H^s(\R)}+\|u(t)\|_{H^s(\R)}+\int_0^t\|u(t')\|_{H^s(\R)}dt'
\end{equation}
for all $t$.
	\end{theorem}

The  smoothing results of this type were first obtained by Linares and Scialon \cite{Linares} for the generalized Korteweg-de Vries equation on the  line $\R$. After Bourgain \cite{Bourgain2} obtained  for the cubic NLS on the $\R^2$. A generalization of this result in $\R^n$ was obtained by Keraani, and Vargas \cite{ker}, and \cite{Compaan} for the cubic NLS on $\R$. Also, Erdogan and Tzirakis \cite{ET} obtained a gain of the regularity for the classical NLS on the torus. Finally, important contributions for the models posed on the half-line was given by Tzirakis et. al (\cite{CT,tzirakis,ET,ET2}).

\subsection{Setting of the problem} In this work, we study the following initial boundary value problem  (IBVP) on the positive half-line
\begin{equation}\label{frach}
\begin{cases}
i\partial_tu(x,t)+(-\Delta)^{\alpha/2}u(x,t) =\lambda|u(x,t)|^2u(x,t),& (x,t)\in\mathbb{R}^+\times(0,T),
\\[5pt]
u(x,0)=u_0(x),& x\in\mathbb{R}^+,\\
u(0,t)=f(t),& t\in (0,T),
\end{cases}
\end{equation}
 where the nonlocal operator $(-\Delta)^{\alpha/2}$,  is defined by
 \begin{equation}
 (-\Delta)^{\alpha/2}v(x)=\int_{\R} e^{ix \xi} |\xi|^{\alpha} \hat{\widetilde{v}}(\xi)d\xi\ 
 \end{equation}
 and $$\widetilde v=\begin{cases}
 v(x),& \text{for}\ x\geq 0,\\
 v(-x),& \text{for}\  x<0.
 \end{cases}$$

It is well-known by Kenig, Ponce and Vega \cite{KPV} that the local smoothing effect for the free linear $e^{it(-\Delta)^{\alpha/2}}$
group operator that solves the associated for the linear fractional Schr\"odinger equation $$i\partial_tu(x,t)+(-\Delta)^{\alpha/2}u(x,t)=0$$
on the all line $\R$, 
$$
\Vert \varphi(t) e^{it(-\Delta)^{\alpha/2}}\phi(x)\Vert_{C(\R_x;H^{\frac{2s-1+\alpha}{2\alpha}}(\R_t))}\leq c \Vert \phi \Vert_{H^s(\R)}.
$$

This motivates the study of IBVP \eqref{frach} in the following setting
\begin{equation}\label{initial}
u_0\in H^s(\R^+)\ \text{and}\ f \in H^{\frac{2s-1+\alpha}{2\alpha}}(\R^+)
\end{equation}

As far as we know this problem never was studied on the half-line. Thus, in this work, we are interested in the following questions for the IBVP \eqref{frach}:

\vspace{0.2cm}
\begin{itemize}
\item \noindent \textit{Is the IBVP \eqref{frach} local well-posedness  in the low regularity Sobolev space?}

\item \textit{Is there some smoothing effect for the IBVP \eqref{frach} similar of the IVP \eqref{frac} context of full line?}
\end{itemize}

Also, as a consequence of the study of these two questions, we can solve the following question for the IVP posed in full line:

\begin{itemize}
	\item \noindent \textit{Does hold  the result of Theorem \ref{theorem0}  for more low regularities?}
\end{itemize}

\vspace{0.2cm}

Thus, now we are able to present the main goal in this paper: to answer these questions, that is, to show  local well-posedness of \eqref{frach} in the low regularity Sobolev spaces, more precisely in $H^s(\R^+)$, for $\frac{2-\alpha}{4}< s<\frac{\alpha-1}{2}$.

We state the main theorem for IBVP \eqref{frach} as follows.
\begin{theorem}\label{theorem1}
Fix $\alpha\in(\frac43,2)$	Let $s \in (\frac{2-\alpha}{4},\frac{\alpha-1}{2})$. For given initial-boundary data $u_0$ and $f$ satisfying \eqref{initial} there exist a positive time $T:=T\left(\|u_0\|_{H^s(\mathbb{R}^+)},\|f\|_{H^{\frac{2s-1+\alpha}{2\alpha}}(\mathbb{R}^+)}\right)$ and a unique solution $u(x,t) \in C((0 , T);H^s(\R^+))$ of the IBVP \eqref{frach}, satisfying
	\[u \in C\bigl(\mathbb{R}^+;\; H^{\frac{2s-1+\alpha}{2\alpha}}(0,T)\bigr) \cap X^{s,b}((0,T) \times \R^+),\]
	for some $b(s) < \frac12$.  Moreover, the map $(u_0,f)\mapsto u$ is analytic from $H^s(\mathbb{R}^+)\times H^{\frac{2s-1+\alpha}{2\alpha}}(\mathbb{R}^+)$ to  $C\big((0,T);\,H^s(\mathbb{R}^+)\big)$.
	
	 Moreover, for $a<\text{min}\{\frac{\alpha-1}{2},\frac{4s+\alpha-2}{2}\}$ holds
	 \begin{equation}\label{regularity}
	u(x,t)-L_{u_0,f}(x,t)\in H^{s+a}(\R^+),
	\end{equation} 
	for all $t$ in $(0,T)$, where $L_{u_0,f}(x,t)$ denotes the solution  of the corresponding linear IBVP \eqref{frach} with $\lambda=0$  \eqref{frach}.
\end{theorem}

Also, as the consequence of the proofs of  Theorem \ref{theorem1}, we have the following result that completes in some sense the recent result of Erdogan, Gurel and Tzirakis \cite{tzirakis2}  described in Theorem \ref{theorem0}.  
\begin{corollary}\label{theorem2}
	Consider the IVP \eqref{frac} on $\mathbb{R}$. Fix $\alpha \in (1,2)$. For any $\frac{2-\alpha}{4}<s<\frac{1}{2}$ and $a<2\alpha-1$, we have that the solution of \eqref{frac}  satisfies 
	\begin{equation}
	u(t)-e^{it(-\Delta)^{\alpha/2}}u_0\in H^{s+a}(\R)
	\end{equation}
	for all $t$ in the maximal interval of existence, where $e^{it(-\Delta)^{\alpha/2}}$ denotes the linear group that solves the linearized IVP \eqref{frac} with $\lambda=0$ . 
	\end{corollary}

	\begin{remarks}Finally, the following comments are now given in order:
		\begin{itemize}
			\item[1.]The proof of Theorem \ref{theorem1} is based on the Fourier restriction method for a suitable extension of solutions for all line $\R$. We first convert the IBVP of \eqref{frach} posed in $\R^+ \times \R^+$ to a the initial value problem (IVP) (integral equation formula) in the whole space $\R \times \R$ (see Section \ref{boundary}) by using the Duhamel boundary forcing operator of Colliander and Kenig (\cite{CK}). The energy and nonlinear estimates (will be established in Sections \ref{linear}) allow us to apply the Picard iteration method for the extended problem, and hence we can complete the proof.
			\medskip
			
			\item[2.] The news ingredients here are the Duhamel boundary forcing operator for the linearized equation and its analysis  (see Section \ref{boundary}), and the crucial trilinear estimates on the Borgain spaces (see Lemma \ref{trilinear}). This estimates is more complicated if compared with the more known nonlinear estimates  for nonlinear dispersive with integer order, since  for in the estimate of Bourgain spaces associated to fractional NLS equation  does not have a exact resonance relation for the frequencies. Also the need of to work with Bougain spaces $X^{s,b}$ with $b<\frac12$ brings more technical difficulties in the proof of this estimate.

			\item[3.] Note that in context of IVP posed in $\R$ the critical regularity is  $s=\frac{2-\alpha}{4}$, then our results on Theorem \ref{theorem1} in context of half-lines are almost sharp with respect the low regularity index. Also note that the upper bound $s<\frac{\alpha-1}{2}$ occurs because of the estimate for the new boundary forcing operator $\mathcal{L}$ give us this bound. Also the bound $\frac{2-\alpha}{4}<s<\frac{\alpha-1}{2}$, give us the restriction $\frac43<\alpha<2$ in our principal result. 
			
					\medskip
						\item[4.] Corollary \ref{theorem2}, in some sense, completes the  gain of regularity obtained by \cite{tzirakis2} for the IVP \ref{frac} posed on the all line $\R$ for few regular regularities, more precisely for $\frac{2-\alpha}{4}<s<\frac{1}{2}$.
				\medskip	
		\end{itemize}
	\end{remarks}

\section{Notations and function spaces}\label{notations}
 For $\phi\in S(\mathbb{R})$, we will define the Fourier transform of $\phi$ by  
$
\hat{\phi}(\xi)=\int_\R e^{-i\xi x}\phi(x)dx
$
and the  inverse transform  by
$
\check{\phi}(\xi)=\dfrac{1}{2\pi}\int_\R e^{i\xi x}\phi(x)dx
$.
For $s\in \R$  the fractional Sobolev space is denoted by
$
H^s(\R)=\left\lbrace  u\in \mathcal{S}':\int_\R(1+|\xi|^2)^s|\hat{u}(\xi)|^2d\xi<\infty
\right\rbrace
$
 with norm 
$
\Vert u \Vert_{H^s(\R)}^2=\int_\R(1+|\xi|^2)^s|\hat{u}(\xi)|^2d\xi
$
for all $u\in H^s(\R)$. The homogeneous Sobolev spaces is defined by $
\dot{H}^s(\R)=\left\lbrace  u\in \mathcal{S}':\int_\R|\xi|^{2s}|\hat{u}(\xi)|^2d\xi<\infty
\right\rbrace
$
with norm
$
\Vert u \Vert_{\dot{H}^s(\R)}^2=\int_\R|\xi|^{2s}|\hat{u}(\xi)|^2d\xi
$
for all $u\in \dot{H}^s(\R)$. For $s\geq 0$ define $\phi \in H^s(\mathbb{R}^+)$ if exists $\tilde{\phi}$ such that $\phi(x)=\tilde{\phi}(x)$ for $x>0$, in this case we set $\|\phi\|_{H^s(\mathbb{R}^+)}=\inf_{\tilde{\phi}}\|\tilde{\phi}\|_{H^{s}(\mathbb{R})}$. For $s\geq 0$ define $$H_0^s(\mathbb{R}^+)=\{\phi \in H^{s}(\mathbb{R}^+);\,\text{supp}\ \phi\ \subset[0,\infty) \}.$$  Also define the $C_0^{\infty}(\mathbb{R}^+)=\{\phi\in C^{\infty}(\mathbb{R});\, \text{supp}\ \phi\subset [0,\infty)\}$ and $C_{0,c}^{\infty}(\mathbb{R}^+)$ as those members of $C_0^{\infty}(\mathbb{R}^+)$ with compact support. We remark that $C_{0,c}^{\infty}(\mathbb{R}^+)$ is dense in $H_0^s(\mathbb{R}^+)$ for all $s\in \mathbb{R}$. Throughout the paper, we fix a cutoff function $\psi \in C_0^{\infty}(\mathbb{R})$ such that $\psi(t)=1$ if $t\in[-1,1]$ and supp $\psi\subset [-2,2]$ and $\psi_T(t)=T\psi(t/T)$. For $s,b\in \mathbb{R}$, and $\alpha\in(1,2)$ we introduce the classical Bourgain spaces $X^{s,b}$  related to the fractional Schr\"odinger equation as the completion of $S'(\mathbb{R}^2)$ under the norms
$
\|u\|_{X^{s,b}}=\left(\int\int \langle \xi\rangle^{2s} \langle \tau-|\xi|^{\alpha}\rangle^{2b} |\hat{u}(\xi,\tau)|^2d\xi d\tau\right)^{\frac{1}{2}}
$

\medskip
The following lemma states elementary properties of the Sobolev spaces. For the proofs we refer the reader \cite{CK}.

\begin{lemma}\label{sobolevh0}
	Let $s\in\R$, the we have
	\begin{itemize}
	\item[(i)] For each $f\in H^s(\mathbb{R})$ with $s\in(-\frac{1}{2},\frac{1}{2})$, we have
$$
\|\chi_{(0,+\infty)}f\|_{H^s(\mathbb{R})}\leq c \|f\|_{H^s(\mathbb{R})}.
$$
	\item[(ii)] For each $f\in H^s(\mathbb{R}^+)$, with $s\in(-\frac{1}{2},\frac{1}{2})$ such that supp$ f\subset [0,1] $. Then 
		\begin{equation}
	\|f\|_{\dot{H}_0^s(\mathbb{R}^+)}\sim \|f\|_{H_0^s(\mathbb{R}^+)}.\nonumber
	\end{equation}
	\item[(iii)] For each $f\in  H_0^s(\mathbb{R}^+)$. Then 
	$$
	\|\psi f\|_{H_0^s(\mathbb{R}^+)}\leq c \|f\|_{H_0^s(\mathbb{R}^+)}.
	$$
	\item[(iv)] For each $f\in \dot{H}^s(\mathbb{R})$, with $s\in(0,\frac{1}{2})$
	$$
	\|\psi f\|_{H^s(\mathbb{R})}\leq c \|f\|_{\dot{H}^{s}(\mathbb{R})},
	$$
	where $c$ only depends of $s$ and $\psi$.
	\end{itemize}
	\end{lemma}


The following lemma concerns Bourgain spaces can be found in \cite{Tao} and \cite{GVT}.
\begin{lemma}\label{gvt2}
	Let $\psi(t)$ be a Schwartz function in time. Then, we have 
	\begin{itemize}
\item[(i)] $\norm{\psi(t)f}_{X^{s,b}} \lesssim_{\psi,b} \norm{f}_{X^{s,b}}$;
\item[(ii)]	$\|\psi_{T}w\|_{ X^{s,b'}}\leq c T^{b-b'}\|w\|_{ X^{s,b}},\ \text{for}\ -\frac{1}{2}< b'< b\leq 0,\ \text{or},\ 0\leq b'<b<\frac{1}{2}$.
	\end{itemize}
\end{lemma}

\subsection{Riemman-Liouville fractional integral}
The tempered distribution $\frac{t_+^{\alpha-1}}{\Gamma(\alpha)}$ is defined as a locally integrable function for Re $\alpha>0$ by
$
\left \langle \frac{t_+^{\alpha-1}}{\Gamma(\alpha)},\ f \right \rangle=\frac{1}{\Gamma(\alpha)}\int_0^{+\infty} t^{\alpha-1}f(t)dt.
$
For Re $\alpha>0$, we have that
\begin{equation}\label{derivada}
\frac{t_+^{\alpha-1}}{\Gamma(\alpha)}=\partial_t^k\left( \frac{t_+^{\alpha+k-1}}{\Gamma(\alpha+k)}\right),\ \text{for all}\ k\in\mathbb{N}.
\end{equation}
 This expression can be used to extend the definition, in the sense of distributions) of $\frac{t_+^{\alpha-1}}{\Gamma(\alpha)}$ to all $\alpha \in \mathbb{C}$. A change of contour  calculation shows that
$
\left(\frac{t_+^{\alpha-1}}{\Gamma(\alpha)}\right)^{\widehat{}}(\tau)=e^{-\frac{1}{2}\pi\alpha}(\tau-i0)^{-\alpha},$
where $(\tau-i0)^{-\alpha}$ is the distributional limit. 

If $f\in C_0^{\infty}(\mathbb{R}^+)$, we define
$
\mathcal{I}_{\alpha}f=\frac{t_+^{\alpha-1}}{\Gamma(\alpha)}*f.
$
It follows that,
  for Re $\alpha>0$, $$\mathcal{I}_{\alpha}f(t)=\frac{1}{\Gamma(\alpha)}\int_0^t(t-s)^{\alpha-1}f(s)ds,\;\;\;\;t>0.$$
Now, for Re $\alpha\leq0$, there exist $\tau\in\N\cup\{0\}$ such that $-\tau\leq Re\alpha<-\tau+1$, then from \eqref{derivada} we can define
\begin{equation*}
\mathcal{I}_{\alpha}f(t)=\frac{1}{\Gamma(\alpha+\tau+1)}\int_0^t(t-s)^{\alpha+\tau}f^{(\tau+1)}(s)ds, \;\;\;\;t>0
\end{equation*}
where $f^{(\tau+1)}$ denotes the $(\tau+1)$-derivative of $f$.

We can extend the definition of $\mathcal{I}_{\alpha}f$ for $t\leq0$, as follow (without of lost generality we use the same notation)
\begin{equation}
\mathcal{I}_{\alpha}f(t)=\begin{cases}
\frac{1}{\Gamma(\alpha)}\int_0^t(t-s)^{\alpha-1}f(s)ds,&t>0
\\[5pt]
0,& t\leq 0.
\end{cases}
\end{equation}
Moreover,
$\mathcal{I}_0f=f,\ \mathcal{I}_1f(t)=\int_0^tf(s)ds,$ and $\mathcal{I}_{-1}f=f'$. Also $\mathcal{I}_{\alpha}\mathcal{I}_{\beta}=\mathcal{I}_{\alpha+\beta}.$
The following lemma, whose proofs can be found in \cite{Holmerkdv}, state some useful properties of the Riemman-Liouville fractional integral operator.
\begin{lemma}\label{estiR-L}
	If $f\in C_0^{\infty}(\mathbb{R}^+)$, then 
	\begin{itemize}
	\item[(i)] For all $\alpha \in \mathbb{C}$, we have $\mathcal{I}_{\alpha}f\in C_0^{\infty}(\mathbb{R}^+)$.
	
	\medskip
	\item[(ii)] For all $0\leq \alpha <\infty$ and $s\in\R$, then 
	$$
	\|\mathcal{I}_{-\alpha}h\|_{H_0^s(\mathbb{R}^+)}\leq c \|h\|_{H_0^{s+\alpha}(\mathbb{R}^+)}.
	$$

	\item[(iii)] For all $0\leq \alpha <\infty$, $s\in \mathbb{R}$ and $\mu\in C_0^{\infty}(\mathbb{R})$, then
$$
\|\mu\mathcal{I}_{\alpha}h\|_{H_0^s(\mathbb{R}^+)}\leq c \|h\|_{H_0^{s-\alpha}(\mathbb{R}^+)},
$$
 where $c=c(\mu)$.
	\end{itemize}
\end{lemma}
For more details on the distribution $\frac{t_+^{\alpha-1}}{\Gamma(\alpha)}$ see \cite{Fr}.

\subsection{Elementary integral inequalities}  In this section, we describe some basic integral inequalities in order to prove the principal trilinear estimate (see Section \ref{trilinear}).

\begin{lemma}[Lemma 3.197 in \cite{Raf}]\label{Lemma-1}
	Let $a,\,b\in[0,\infty)$ and $s\geq0$. Then there exist positive constants $m_s$ and $M_s$ depending only on $s$ such that
	$$
	m_s(a^s+b^s)\leq(a+b)^s\leq M_s(a^s+b^s).
	$$
\end{lemma}

In a particular case $0\leq s\leq1$ we have that $M_s\leq 1$, which is important to prove the next lemma.
\begin{lemma}\label{Lema0}
	Let $m,\,n\in\R$ such that $m\leq n$ and $0<c<1$, then
	$$
	\int_m^n \dfrac{1}{\langle x\rangle^c}\,dx\lesssim \langle n-m\rangle^{1-c}
	$$
\end{lemma}
\begin{proof}
	Suppose that $0\leq m$, then by Lemma \ref{Lemma-1}
	$$
	\int_m^n \dfrac{1}{\langle x\rangle^c}\,dx=\dfrac{1}{1-c}\left[(n+1)^{1-c}-(m+1)^{1-c}\right]\lesssim \langle n-m\rangle^{1-c}.
	$$

	On the other hand, for $m=0$ we have
	$
	\int_0^n \dfrac{1}{\langle x\rangle^c}\,dx\lesssim \langle n\rangle^{1-c}.
	$
	
	The case $n\leq 0$ is analogous to the first case. 
	
	Finally, in the general case $m\leq0\leq n$, we get
	$$
	\begin{aligned}
	\int_m^n \dfrac{1}{\langle x\rangle^c}\,dx=\int_m^0 \dfrac{1}{\langle x\rangle^c}\,dx+\int_0^n \dfrac{1}{\langle x\rangle^c}\,dx
	&=\int_0^{-m} \dfrac{1}{\langle x\rangle^c}\,dx+\int_0^n \dfrac{1}{\langle x\rangle^c}\,dx
	\\[8pt]
	&\lesssim \langle -m\rangle^{1-c}+\langle n\rangle^{1-c}\lesssim\langle n-m\rangle^{1-c},
	\end{aligned}
	$$
	where in the last inequality we use  Lemma \ref{Lemma-1}, thus we conclude the  proof of lemma.
\end{proof}
\begin{lemma}[Lemma 4.2 in \cite{GVT}]\label{Lemma1}
	Let  $\beta\geq\gamma\geq0$, such that $\beta+\gamma>1$, then
	\begin{equation*}
	\int\frac{1}{\langle x-a_1\rangle^{\beta}\langle x-a_2\rangle^{\gamma}}dx\lesssim
	\langle a_1-a_2\rangle^{-\gamma}\Phi_{\beta}(a_1-a_2),
	\end{equation*}
	where
	\begin{equation*}
	\Phi_\beta(a)\backsim
	\begin{cases}
	1&\beta>1
	\\
	\log(1+\langle a\rangle) & \beta=1
	\\
	\langle a\rangle^{1-\beta} &\beta<1
	\end{cases}
	\end{equation*}
\end{lemma}
\begin{lemma}[Lemma 6.3 in \cite{tzirakis}]\label{Lemma2}
	For fixed $\rho\in(\dfrac{1}{2},1)$, we then
	\begin{equation*}
	\int\frac{1}{\langle x\rangle^{\rho}|x-a|^{\frac{1}{2}}}dx\lesssim\frac{1}{\langle a\rangle^{\rho-\frac{1}{2}}}.
	\end{equation*}
\end{lemma}
\begin{lemma}[Lemma 2 in \cite{Demirbas}]\label{Lemma3}
	Fix $\alpha\in(1,2)$. For $n,j,k\in \R$, we have
	$$
	\big||k+n|^\alpha-|k+m+n|^\alpha+|k+m|^\alpha-|k|^\alpha\big|\gtrsim \dfrac{|m||n|}{(|m|+|n|+|k|)^{2-\alpha}}.
	$$
	where the implicit constant depends of $\alpha$.
\end{lemma}

Finally, the following Lemma can be found in the proof of Lemma 6.2 in \cite{Cavalcante}.
\begin{lemma}\label{Lemadesigual}
Let $b<\frac{1}{2}$, $-1<\lambda<\frac{1}{2}$ and $-1<\sigma+\lambda<0$.
	$$
	\int_{\R}\frac{|\eta|^{\lambda}(1+|\eta|)^{\sigma}}{(1+|\tau-\eta|)^{2-2b}} \ d\eta\leq c (1+|\tau|)^{\sigma+\lambda}.
	$$
\end{lemma}

\section{Linear Version}\label{linearversion} 
We define the unitary group associated to the linear fractional Schr\"odinger equation as
$$
v(t):=e^{it(-\Delta)^{\alpha/2}}\phi=\dfrac{1}{2\pi}\int_\R e^{ix\xi}e^{it|\xi|^\alpha}\hat{\phi}(\xi)d\xi
$$
where $\alpha\in(0,2)$. Follow that $v$ is a solution of  
\begin{equation}\label{linear}
\begin{cases}
i\partial_tv(x,t)+(-\Delta)^{\alpha/2}v(x,t) =0,& (x,t)\in\mathbb{R}\times\mathbb{R},
\\[5pt]
v(x,0)=\phi(x),& x\in\mathbb{R}.
\end{cases}
\end{equation}


\medskip
The following Lemma states some inequalities of $v$, which are important in the proof of the Theorem \ref{theorem1}.
\begin{lemma}\label{grupo}
Let $0\leq s<\infty$ and $\psi$ is a cutoff function as defined above. If $\phi\in H^s(\R)$, then
\begin{itemize}
\item[$(a)$] $($Space trace$)$ Let $\alpha\in(1,2)$
$$
\Vert e^{it(-\Delta)^{\alpha/2}}\phi(x)\Vert_{C(\R_t;H^s(\R_x))}\leq  \Vert \phi \Vert_{H^s(\R)},
$$
\item[$(b)$] $($Time space$)$ Let $\alpha \in (1,2)$
$$
\Vert \psi(t) e^{it(-\Delta)^{\alpha/2}}\phi(x)\Vert_{C(\R_x;H^{\frac{2s-1+\alpha}{2\alpha}}(\R_t))}\leq c \Vert \phi \Vert_{H^s(\R)},
$$
where $c$ is a function which depend of $\psi$.

\item[$(c)$] $($Bourgain space$)$ Let $b\in(0,1)$
$$
\Vert \psi(t) e^{it(-\Delta)^{\alpha/2}}\phi(x)\Vert_{X^{s,b}}\leq c \Vert \psi\Vert_{H^1(\R)}\Vert \phi \Vert_{H^s(\R)}.
$$
\end{itemize}
\end{lemma}
\begin{proof}
The proof of (a) follows directly from 
$e^{it(-\Delta)^{\alpha/2}}\phi(x)= \mathcal{F}^{-1}\left\lbrace e^{it|\xi|^\alpha}\hat{\phi}(\xi)\right\rbrace(x)$, for all $t$. Now, to prove (b) we write
\begin{equation}
\begin{aligned}
e^{it(-\Delta)^{\alpha/2}}\phi(x)&=\int_\R e^{ix\xi}e^{it|\xi|^{\alpha}}\hat{\phi}(\xi)d\xi
=\int_0^{\infty}e^{ix\xi}e^{it\xi^{\alpha}}\hat{\phi}(\xi)d\xi+\int^0_{-\infty}e^{ix\xi}e^{it(-\xi)^{\alpha}}\hat{\phi}(\xi)d\xi
\\[5pt]
&= \mathrm{I}+\mathrm{II}
\end{aligned}
\end{equation}

First, we study $\mathrm{I}$. Using change of variable $\xi'=|\xi|^\alpha$ we obtain
$$
\begin{aligned}
\mathrm{I}=\dfrac{1}{\alpha}\int_0^{\infty}\xi^{\frac{1-\alpha}{\alpha}} e^{ix\xi^{1/\alpha}}e^{it\xi}\hat{\phi}(\xi^{1/\alpha})d\xi
=\dfrac{1}{\alpha}\int_\R e^{it\xi}\left\lbrace\mathcal{X}_{\R^+}(\xi)\xi^{\frac{1-\alpha}{\alpha}} e^{ix\xi^{1/\alpha}}\hat{\phi}(\xi^{1/\alpha})\right\rbrace d\xi,
\end{aligned}
$$
hence
\begin{equation*}
\Vert \mathrm{I} \Vert^2_{\dot{H}^{\frac{2s-1+\alpha}{2\alpha}}(\R_t)}\leq
\int_0^{\infty}\dfrac{1}{\alpha^2}|\xi|^{\frac{2s+1-\alpha}{\alpha}}|\hat{\phi}(\xi^{1/\alpha})|^2d\xi
=\int_0^{\infty}\dfrac{1}{\alpha}|\xi|^{s}|\hat{\phi}(\xi)|^2d\xi.
\end{equation*}

\medskip
A similar estimate can be done to $\mathrm{II}$. consequently we obtain
\begin{equation}
\Vert e^{it(-\Delta)^{\alpha/2}}\phi(x)\Vert_{\dot{H}^{\frac{2s-1+\alpha}{2\alpha}}(\R_t)}\leq c \Vert \phi \Vert_{\dot{H}^s(\R)},\label{linertrace1}
\end{equation}
for all $s\in\R$. 

On the other hand, using  Lemma \ref{sobolevh0} (iv), and \eqref{linertrace1} 
\begin{equation}\label{lineartrace2}
\begin{aligned}
\Vert\varphi(t) e^{it(-\Delta)^{\alpha/2}}\phi(x)\Vert_{L^2(\R_t)}&\leq\,\Vert\varphi(t) e^{it(-\Delta)^{\alpha/2}}\phi(x)\Vert_{H^{\frac{\alpha-1}{2\alpha}}(\R_t)}
\\[5pt]
&\leq c \Vert e^{it(-\Delta)^{\alpha/2}}\phi(x)\Vert_{\dot{H}^{\frac{\alpha-1}{2\alpha}}(\R_t)}\leq c \Vert\phi\Vert_{L^2(\R_t)}
\end{aligned}
\end{equation}
Finally combining \eqref{linertrace1} and \eqref{lineartrace2}, we obtain the result.

\medskip
Finally, we prove (c). To do this we 
first observe
$$
\mathcal{F}_{x,t}\left\lbrace\varphi(t)e^{it(-\Delta)^{\alpha/2}}\phi(x)\right\rbrace(\xi,\tau)=\mathcal{F}_t\{\varphi\}(\tau-|\xi|^{\alpha})\cdot\mathcal{F}_x\{\phi\}(\xi).
$$
then we have from the definition of $X^{s,b}$
$$
\begin{aligned}
\Vert \varphi(t) e^{it(-\Delta)^{\alpha/2}}&\phi(x)\Vert_{X^{s,b}}\\&= \int_\R\int_\R(1+|\xi|)^{2s}(1+|\tau-|\xi|^\alpha|)^{2b}|\mathcal{F}_t\{\varphi\}(\tau-|\xi|^{\alpha})\cdot\mathcal{F}_x\{\phi\}(\xi)|^2d\xi d\tau
\\[5pt]
&= \left[\int_\R(1+|\tau|)^{2b}|\hat{\varphi}(\tau)|^2d\tau\right]\left[\int_\R(1+|\xi|)^{2s}|\hat{\phi}(\xi)|^2d\xi\right]
\\[5pt]
&=\Vert \varphi \Vert_{H^b(\R)}\cdot\Vert \phi\Vert_{H^s(\R)}
\end{aligned}
$$
This finishes the proof of Lemma \ref{grupo}.
\end{proof}
\section{The Duhamel Boundary Forcing Operator}\label{boundary}

In this section, we introduce an adaptation of the Duhamel boundary forcing operator in context of the  fractional  Schr\"odinger equation. The principal idea of the Colliander and Kenig is to solve the following IVP on all line
\begin{equation}
\label{FTPME} 
\left \{
\begin{aligned}
&i\partial_t u+(-\Delta)^{\alpha/2}u=\frac{2\pi}{B(0)\Gamma(1-\frac{1}{\alpha})}\delta_0(x) \ \mathcal{I}_{\frac{1}{\alpha}-1} f(t), \quad (x,t)\in \R\times(0,T),
\\[5pt]
&u(x,0) = 0,\quad x \in \R,
\end{aligned}
\right.
\end{equation}
where
$
B(x)=\int_\R e^{ix\xi} \ e^{i|\xi|^{\alpha}} \ d\xi.
$
An application of  van der Corputs Lemma proves that $B$, for $1<\alpha<2$, is a well defined function and the proof is standard.

By using the Duhamel formula, we define for any $f\in C_0^{\infty}(\R^n)$
\begin{equation}\label{delta}
\begin{aligned}
\mathcal{L}f(x,t)&=\frac{2\pi}{B(0)\Gamma(1-\frac{1}{\alpha})}\int_0^te^{i(t-t')(-\Delta)^{\alpha/2}}\delta_0(x) \  \mathcal{I}_{\frac{1}{\alpha}-1}f(t') \ dt'
\\[7pt]
&=\frac{1}{B(0)\Gamma(1-\frac{1}{\alpha})}\int_0^t \int_\R e^{ix\xi} \ e^{i(t-t')|\xi|^{\alpha}} \mathcal{I}_{\frac{1}{\alpha}-1}f(t') \ d\xi dt'
\end{aligned}
\end{equation}
which is the solution of \eqref{FTPME}. Using the change of variables $(t-t')|\xi|^\alpha=|\xi'|^\alpha$ we obtain the following representation of $\mathcal{L}f$
\begin{equation}\label{formula}
\begin{aligned}
\mathcal{L}f(x,t)&=\frac{1}{B(0)\Gamma(1-\frac{1}{\alpha})}\int_0^t  (t-t')^{-1/\alpha}  B\left(\dfrac{x}{(t-t')^{1/\alpha}}\right)\mathcal{I}_{\frac{1}{\alpha}-1}f(t') \ dt'
\end{aligned}
\end{equation}
The following Lemma states some continuity properties of the function $\mathcal{L}f(x,t)$.

\begin{lemma}\label{continuidade}
Let $f\in C^{\infty}_{0,c}(\R^+)$ and $\alpha\in(1,2)$.
\begin{itemize}
\item[(i)] For fixed $t$, $\mathcal{L}f(x,t)$ is continuous in $x$ for all $x\in\R$.  
\\
\item[(ii)] For $N, k$ nonnegative integers and fixed $x$, $\partial^k_t\mathcal{L}f(x,t)$ is continuous in $t$ for all $t\in\R^+$. We also have the pointwise estimates, on $[0,T]$,
$$
|\partial^k_t\mathcal{L}f(x,t)|+|\partial_x(-\Delta)^{\alpha/2-1}\mathcal{L}f(x,t)|\leq c\langle x \rangle^{-N}
$$
where $c=c(f,N,k,T)$. 
\end{itemize}
\end{lemma}
\begin{proof}
The continuity of $\mathcal{L}f(x,t)$ follows from formula \eqref{formula} and the convergence of dominated Theorem. Now, we prove (ii). Let $h=\mathcal{I}_{\frac{1}{\alpha}-1}f$ and $\phi(\xi, t)=\int_{0}^{t} e^{-i\left(t-t^{\prime}\right) \xi} h\left(t^{\prime}\right) d t^{\prime}$. Integrating by parts in $t'$ we have that $|\partial_x^k\phi(\xi,t)|\leq c \langle \xi \rangle^{-k-1}$, where $c$ depends of $k$, $h$ and $t$. Thus by using chule rule $\left|\partial_{\xi}^{k} \phi\left(|\xi|^{\alpha}, t\right)\right| \leq c\langle\xi\rangle^{(-k-1)\alpha+(\alpha-1)k}=c\langle\xi\rangle^{-\alpha-k}$. As $\mathcal{L} f(x, t)=\int_{\xi} e^{i x \xi} \phi\left(|\xi|^{\alpha}, t\right) d \xi$, an integration by parts give us
\begin{equation}
|\mathcal{L} f(x, t)| \leq c\langle x\rangle^{-N}.
\end{equation}
By using $\partial_{t}\left[e^{i\left(t-t^{\prime}\right) (-\Delta)^{\alpha/2}} \delta_{0}(x)\right]=-\partial_{t^{\prime}}\left[e^{i\left(t-t^{\prime}\right) (-\Delta)^{\alpha/2}} \delta_{0}(x)\right]$ and integrating by parts \eqref{delta} we obtain that $\partial_{t} \mathcal{L} f=\mathcal{L} \partial_{t} f$. It follows that, for fixed $x$, $\partial_{t}^{k} \mathcal{L} f(x, t)$ is continuous in $t$ and $\left|\partial_{t}^{k} \mathcal{L} f(x, t)\right| \leq c\langle x\rangle^{-N}$.
\end{proof}

By using the previous lemma and \eqref{formula} we have that  \begin{equation}\label{traco}\lim_{x\rightarrow 0}\mathcal{L}f(x,t)=\mathcal{L}f(x,t)=f(t).\end{equation} Now, we able to solve a linearized IBVP on the hal-line. More precisely, by combining \eqref{traco} with  \eqref{FTPME} we get that the function $u(x,t)=\mathcal{L}f(x,t)\big|_{\{x>0\}}$ solves the linearized IBVP
\begin{equation}
\left\{\begin{array}{ll}{i \partial_{t} u(t, x)+(-\Delta)^{\alpha / 2} u( x,t)=0,} & {( x,t) \in \mathbb{R^+} \times(0,T)} 
\\
 {u(0, x)=0,} & {x \in \mathbb{R}^{+}}
\\
 {u(t, 0)=f(t),} & {t \in(0, T)}\end{array}\right.
\end{equation}

\medskip
Now we state the needed estimates for the Duhamel boundary forcing 
operators class.
\begin{lemma}\label{boundaryestimate}Let  $\alpha\in(1,2)$ and $\psi$ is a cutoff function as defined above . Then
\begin{itemize}
\item[$(a)$]$($Space traces$)$  If $-\frac{1}{2}<s<\frac{2\alpha-1}{2}$, then
	$$
	\|\mathcal{L}f(x,t)\|_{C\big(\mathbb{R}_t;\,H^s(\mathbb{R}_x)\big)}\leq c \|f\|_{H_0^\frac{2s+\alpha-1}{2\alpha}(\mathbb{R}^+)}.
	$$
\item[(b)]$($Time traces$)$   If $s\in \R$, then
		$$\|\psi(t)\mathcal{L}f(x,t)\|_{C\big(\mathbb{R}_x;\,H_0^\frac{2s+\alpha-1}{2\alpha}(\mathbb{R}_t^+)\big)}\leq c \|f\|_{H_0^\frac{2s+\alpha-1}{2\alpha}(\mathbb{R}^+)};$$
		
\item[(c)]$($Bourgain spaces$)$ If $0<b<\frac{1}{2}$, $-\frac{1}{2}<s<\frac{\alpha-1}{2}$, then
		$$\|\psi(t)\mathcal{L}f(x,t)\|_{X^{s,b}}\leq c \|f\|_{H_0^\frac{2s+\alpha-1}{2\alpha}(\mathbb{R}^+)}.$$
\end{itemize}
\label{edbf}
\end{lemma}
\begin{remark}

	\begin{itemize}
The assumption $b<\frac{1}{2}$ is crucial in the proof of Lemma \ref{edbf} (c). This fact forced us  to obtain the unknown and more complicated trilinear bourgain  space estimate for $b<1/2$.

\end{itemize}
\end{remark}
\begin{proof}
Initially we prove (a). From the definition of $\mathcal{L}f$, we see that 
$$
\mathcal{F}_x(\mathcal{L}f)(\xi,t)=\frac{2\pi}{B(0)\Gamma(1-\frac{1}{\alpha})}\int_0^te^{i(t-t')|\xi|^{\alpha}}\mathcal{I}_{\frac{1}{\alpha}-1}f(t')dt'.
$$
Since $ \mathcal{F}_x(\mathcal{L}f)(\xi,t)$ is a even function in $\xi$ and  changing of variables $\eta=|\xi|^\alpha$ we obtain, for fixed $t$,
\begin{eqnarray*}
\|\mathcal{L}f(x,t)\|_{H^s(\mathbb{R})}^2&=&\frac{2c}{\alpha} \int_0^{\infty}\eta^{\frac{1}{\alpha}-1}(1+|\eta|^{\frac{2}{\alpha}})^s\left|\int_0^te^{i(t-t')\eta}\mathcal{I}_{\frac{1}{\alpha}-1}f(t')dt'\right|^2d\eta.
\end{eqnarray*}
 Then using Lemma \ref{sobolevh0} (ii), we control the last expression by
\begin{eqnarray*}
		&&\frac{2c}{\alpha} \int_0^{\infty}(1+|\eta|)^{\frac{2s+1-\alpha}{\alpha}}\left|\int_0^te^{-it'\eta}\mathcal{I}_{\frac{1}{\alpha}-1}f(t')dt'\right|^2d\eta
		\\[5pt]
		&&\leq\frac{2c}{\alpha} \int_{\R}(1+|\eta|)^{\frac{2s+1-\alpha}{\alpha}}\left|(\chi_{(0,t)}\mathcal{I}_{\frac{1}{\alpha}-1}f)^{\widehat{}}(\eta)\right|^2d\eta
\end{eqnarray*}
	By Lemmas \ref{sobolevh0} (i) (to remove $\chi_{(-\infty,t)})$ and \ref{estiR-L} (ii) (to estimate $\mathcal{I}_{\frac{1}{\alpha}-1}$), we  bound the last expression by $\|f\|_{H_0^{\frac{2s+\alpha-1}{2\alpha}}(\mathbb{R}^+)}^2$, which proves (a).

Now we obtain (b). By Lemma \ref{sobolevh0} (iii), we can ignore the test function. On the other hand, from the definition of $\mathcal{L}f(x,t)$ we have	
$$
\begin{aligned}
\mathcal{F}_t(\mathcal{L}f)(x,\eta)&=c \hat{f}(\eta)\lim_{\epsilon\to 0}\int_{|\eta-|\xi|^\alpha|>\epsilon}\frac{e^{ix\xi} e^{\frac{\pi i}{2 \alpha}} (\eta-i0)^{1-\frac{1}{\alpha}}}{\eta-|\xi|^\alpha} d\xi
\end{aligned}
$$
where $c=\frac{1}{\alpha B(0)\Gamma(1-\frac{1}{\alpha})} $. It follows that
$$
\begin{aligned}
\Vert \mathcal{L}f(x,\cdot)\Vert_{H^{\frac{2s+\alpha-1}{2\alpha}}_0(\R^+_t)}^2&=\int_{\R}(1+|\eta|)^{\frac{2s+\alpha-1}{\alpha}}|\hat{f}(\eta)|^2\left|\lim_{\epsilon\to 0}\int_{|\eta-|\xi|^\alpha|>\epsilon}\frac{e^{ix\xi} e^{\frac{\pi i}{2 \alpha}} (\eta-i0)^{1-\frac{1}{\alpha}}}{\eta-|\xi|^\alpha} d\xi \right|^2d\eta
\end{aligned}
$$
 Thus, to conclude the proof, we need to show that the function 
$$
g(x,\eta):=\lim_{\epsilon\to 0}\int_{|\eta-|\xi|^\alpha|>\epsilon} \frac{e^{ix\xi}e^{\frac{\pi i}{2 \alpha}}(\eta-i0)^{1-\frac{1}{\alpha}}}{\eta-|\xi|^\alpha} d\xi
$$
is limited. 
	
	Changing of variables $\xi\mapsto |\eta|^{\frac{1}{\alpha}}\xi$ and using $
	e^{\frac{\pi i}{2 \alpha}}(\eta-i0)^{1-\frac{1}{\alpha}}=e^{-\frac{\pi i}{2 \alpha}}\eta_{+}^{1-\frac{1}{\alpha}}+e^{\frac{\pi i}{2\alpha}}\eta_{-}^{1-\frac{1}{\alpha}},
	$	we obtain
	\begin{eqnarray*}
		g(x,\eta)&=&\lim_{\epsilon\to 0}\int_{\left|\eta-|\eta||\xi|^\alpha\right|>\epsilon} \frac{e^{i|\eta|^{\frac{1}{\alpha}}x\xi}|\eta|^{\frac{1}{\alpha}}(e^{-\frac{\pi i}{2 \alpha}}\eta_{+}^{1-\frac{1}{\alpha}} + e^{\frac{\pi i}{2\alpha}}\eta_{-}^{1-\frac{1}{\alpha}})d\xi}{\eta-|\eta||\xi|^\alpha}
		\\
		&=&e^{-\frac{\pi i}{2 \alpha}}\chi_{\{\eta>0\}}\lim_{\epsilon\to 0}\int_{\left|1-|\xi|^\alpha\right|>\epsilon}\frac{e^{i|\eta|^{\frac{1}{\alpha}}x\xi}}{1-|\xi|^\alpha}d\xi + e^{\frac{\pi i}{2 \alpha}}\chi_{\{\eta<0\}}\int_{\R}\frac{e^{i|\eta|^{\frac{1}{\alpha}}x\xi}}{1+|\xi|^\alpha}d\xi.
	\end{eqnarray*}
Since $\alpha>1$ we obtain that $g(x,\eta)$ is uniformly limited in $x$ and $\eta$,  this complete the proof of (b).

\medskip	
Finnaly, we prove (c). From the definition of $\mathcal{L}f$, we see that 
$$
\mathcal{F}_x(\mathcal{L}f)(\xi,t)=c e^{it|\xi|^{\alpha}}G(\xi,t),
$$
where 
$$
G(\xi,t)=\int_0^te^{-it'|\xi|^\alpha}\mathcal{I}_{\frac{1}{\alpha}-1}f(t')dt'
$$

On the other hand
$$
\begin{aligned}
\mathcal{F}_{t,x}\left(\psi(t)\mathcal{L}f(x,t))(\xi,\eta\right)&=\mathcal{F}_{t}\big(\psi(\cdot)\mathcal{F}_x(\mathcal{L}f\big)(\xi,\cdot))(\eta)
\\[5pt]
&=c \ \mathcal{F}_{t}\big(\psi(t)e^{it|\xi|^{\alpha}}G(\xi,t))(\eta)
\\[5pt]
&=c \ \mathcal{F}_{t}\big(\psi(t) G(\xi,t))(\eta-|\xi|^{\alpha})
\end{aligned}
$$
Then observe
$$
\begin{aligned}
\Vert \psi(t)\mathcal{L}f(x,t) \Vert_{X^{s,b}}^2&=\int_{\R}\int_{\R}(1+|\xi|)^{2s}(1+|\eta-|\xi|^{\alpha}|)^{2b}\left|\mathcal{F}_{t,x}(\psi(t)\mathcal{L}f(x,t))(\xi,\eta)\right|^2 d\xi \ d\eta
\\[5pt]
&=c \ \int_{\R}\int_{\R}(1+|\xi|)^{2s}(1+|\eta-|\xi|^{\alpha}|)^{2b}\left|\mathcal{F}_{t}\big(\psi(t) G(\xi,t))(\eta-|\xi|^{\alpha})\right|^2 d\xi \ d\eta
\\[5pt]
&=c \ \int_{\R}(1+|\xi|)^{2s}\left\lbrace\int_{\R}(1+|\tau|)^{2b}\left|\mathcal{F}_{t}\big(\psi(t) G(\xi,t))(\tau)\right|^2 \ d\tau\right\rbrace d\xi
\end{aligned}
$$
By the Lemma \ref{sobolevh0} (iii) we obtain
$$
\begin{aligned}
\Vert \psi(t)\mathcal{L}f(x,t) \Vert_{X^{s,b}}^2&\leq c \ \int_{\R}(1+|\xi|)^{2s}\left\lbrace\int_{\R}(1+|\tau|)^{2b}\left|\mathcal{F}_{t}\big(G(\xi,t))(\tau)\right|^2 \ d\tau\right\rbrace d\xi
\end{aligned}
$$
On the other hand, we know
$$
\mathcal{F}_{t}\big(G(\xi,t))(\eta-|\xi|^\alpha)=\frac{e^{\frac{\pi i}{2 \alpha}}(\eta-i0)^{1-\frac{1}{\alpha}}\hat{f}(\eta)}{\eta-|\xi|^\alpha}
$$

where $
e^{\frac{\pi i}{2 \alpha}}(\eta-i0)^{1-\frac{1}{\alpha}}=e^{-\frac{\pi i}{2 \alpha}}\eta_{+}^{1-\frac{1}{\alpha}}+e^{\frac{\pi i}{2\alpha}}\eta_{-}^{1-\frac{1}{\alpha}}$.	
Consequently	
$$
\begin{aligned}
\Vert \psi(t)\mathcal{L}f(x,t) \Vert_{X^{s,b}}^2&\leq c \ \int_{\R}(1+|\xi|)^{2s}\left\lbrace\int_{\R}(1+|\eta-|\xi|^{\alpha}|)^{2b}\left|\frac{e^{\frac{\pi i}{2 \alpha}}(\eta-i0)^{1-\frac{1}{\alpha}}\hat{f}(\eta)}{\eta-|\xi|^\alpha}\right|^2 \ d\eta\right\rbrace d\xi
\\[5pt]
&\leq c \ \int_{\R}(1+|\xi|)^{2s}\left\lbrace\int_{\R}(1+|\eta-|\xi|^{\alpha}|)^{2b}\frac{|\eta|^{2-\frac{2}{\alpha}}|\hat{f}(\eta)|^2}{(\eta-|\xi|^\alpha)^2} \ d\eta\right\rbrace d\xi
\\[5pt]
&\leq c \ \int_{\R}|\eta|^{2-\frac{2}{\alpha}}\left\lbrace\lim_{\epsilon\to 0}\int_{|\eta-|\xi|^\alpha|>\epsilon}(1+|\eta-|\xi|^{\alpha}|)^{2b}\frac{(1+|\xi|)^{2s}}{(\eta-|\xi|^\alpha)^2} \ d\xi\right\rbrace |\hat{f}(\eta)|^2d\eta
\end{aligned}
$$
Thus, it suffices to obtain
\begin{equation}\label{tna}
I(\eta)=\lim_{\epsilon\to 0}\int_{|\eta-|\xi|^\alpha|>\epsilon}(1+|\eta-|\xi|^{\alpha}|)^{2b}\frac{(1+|\xi|)^{2s}}{(\eta-|\xi|^\alpha)^2} \ d\xi\leq c (1+|\eta|)^{\frac{2s+1-\alpha}{\alpha}}
\end{equation}
Observe that	
$$
\begin{aligned}
I(\eta)&\leq\lim_{\epsilon\to 0}\int_{|\eta-|\xi|^\alpha|>\epsilon}(1+|\eta-|\xi|^{\alpha}|)^{2b}\frac{(1+|\xi|)^{2s}}{(1+|\eta-|\xi|^\alpha|)^2} \ d\xi
\\[5pt]
&\leq\int_{\R}\frac{(1+|\xi|)^{2s}}{(1+|\eta-|\xi|^{\alpha}|)^{2-2b}} \ d\xi.
\end{aligned}
$$
Then by the Lemma \ref{Lemadesigual} we obtain
$$
\int_{\R}\frac{(1+|\xi|)^{2s}}{(1+|\eta-|\xi|^{\alpha}|)^{2-2b}}  d\xi\leq  \int_{\R}\frac{|\xi|^{\frac{1-\alpha}{\alpha}}(1+|\xi|)^{2s/\alpha}}{(1+|\eta-\xi|)^{2-2b}}  d\xi \leq c (1+|\eta|)^{\frac{2s+1-\alpha}{\alpha}},
$$	
where we have used that $s<\frac{\alpha-1}{2}$.
This finishes the proof of the lemma.	
\end{proof}
\section{Nonlinear Versions}\label{duhameli}
We define the Duhamel inhomogeneous solution operator $\mathcal{D}$ as
\begin{equation*}
\mathcal{D}w(x,t)=-i\int_0^te^{i(t-t')(-\Delta)^{\alpha/2}}w(x,t')dt',
\end{equation*}
which is a solution for the following problem
\begin{equation}\label{nonlinear}
\left \{
\begin{array}{l}
i\partial_tv(x,t)+(-\Delta)^{\alpha/2}v(x,t) =w(x,t),\ (x,t)\in\mathbb{R}\times\mathbb{R},
\\[5pt]
v(x,0) =0,\ x\in\mathbb{R}.
\end{array}
\right.
\end{equation}
The next result establishes the estimates for the Duhamel inhomogeneous solution operator, the prove is in Appendix \ref{apduhamel}. 

\begin{lemma}\label{duhamel} Let $s\in \mathbb{R}$ and $\psi$ is a cutoff function as defined above, then: 
	\begin{itemize}
		\item[(a)]$($Space traces$)$ If $-\frac{1}{2}<c<0$ , then
		\begin{equation*}
		\|\psi(t)\mathcal{D}w(x,t)\|_{C\big(\mathbb{R}_t;\,H^s(\mathbb{R}_x)\big)}\leq c\|w\|_{X^{s,c}};
		\end{equation*}
		\item[(b)]$($Time traces$)$ If $-\frac{1}{2}<c<0$, then
		 \begin{equation*}
		\|\psi(t)\mathcal{D}w(x,t)\|_{C(\mathbb{R}_x;H^{\frac{2s-1+\alpha}{2\alpha}}(\mathbb{R}_t))}\leq 
		c\|w\|_{X^{s,c}},\ \text{if}\ 0\leq  s\leq\frac{1}{2};
		\end{equation*}
		\item[(c)]$($Bourgain estimates$)$ If $-\frac{1}{2}<c\leq0\leq b\leq c+1$, then
		\begin{equation*}
		\|\psi(t)\mathcal{D}w(x,t)\|_{X^{s,b}}\leq \|w\|_{X^{s,c}}.
		\end{equation*}
	\end{itemize}
\end{lemma}

\section{Trilinear estimate}
In this section, we prove the   crucial trilinear estimate on the Bourgain space for $b<\frac12$. 

 To do this we need of the following more technical lemmas.
\begin{lemma}\label{Lemma4}
	For $\frac{2-\alpha}{4}<s$ and $a<\min\{\frac{\alpha-1}{2},\frac{4s+\alpha-2}{2}\}$ there exist $\epsilon>0$ such that for 
	$\frac{1}{2}-\epsilon<b<\frac{1}{2}$, we have
	\begin{equation}
	\sup_{\xi}\int_{A}\; \dfrac{\langle\xi\rangle^{2s+2a}\langle\xi_1\rangle^{-2s}\langle\xi_2\rangle^{-2s}\langle\xi-\xi_1+\xi_2\rangle^{-2s}}{\langle|\xi|^\alpha-|\xi_1|^\alpha+|\xi_2|^\alpha-|\xi-\xi_1+\xi_2|^\alpha\rangle^{1-6\epsilon}}d\xi_1 d\xi_2,\label{priceqlemma4}
	\end{equation}
	is bounded, where $A:=\{|\xi_1-\xi|<1\, or\, |\xi_1-\xi_2|<1\}$.
\end{lemma}
\begin{proof}
	See Appendix \ref{proof1} for a complete proof.
\end{proof}
\begin{lemma}\label{Lemma5}
	For $\frac{2-\alpha}{4}<s<\frac{1}{2}$ and $a<\min\{\frac{\alpha-1}{2},\frac{4s+\alpha-2}{2}\}$ there exist $\epsilon>0$ such that for 
	$\frac{1}{2}-\epsilon<b<\frac{1}{2}$, we have
	\begin{equation}
	\sup_{\xi}\int_{B}\; \dfrac{\langle\xi\rangle^{2s+2a}\langle\xi_1\rangle^{-2s}\langle\xi_2\rangle^{-2s}\langle\xi-\xi_1+\xi_2\rangle^{-2s}}{\langle|\xi|^\alpha-|\xi_1|^\alpha+|\xi_2|^\alpha-|\xi-\xi_1+\xi_2|^\alpha\rangle^{1-6\epsilon}}d\xi_1 d\xi_2,\label{priceqlemma5}
	\end{equation}
	is bounded, where $B:=\{|\xi_1-\xi|\geq1\, and\, |\xi_1-\xi_2|\geq1\}$.
\end{lemma}
\begin{proof}
	See in Appendix \ref{proof2}.
\end{proof}

Now we obtain the triliner estimate.
\begin{proposition}\label{trilinear}
	For $\frac{2-\alpha}{4}<s$ and $a<\min\{\frac{\alpha-1}{2},\frac{4s+\alpha-2}{2}\}$ there exist $\epsilon>0$ such that for 
	$\frac{1}{2}-\epsilon<b<\frac{1}{2}$, we have
	$$
	\Vert |u|^2 u  \Vert_{X^{s+a,-b}}\lesssim \Vert u \Vert^3_{X^{s,b}}
	$$
\end{proposition}
\begin{proof}
	By writing the Fourier transform of $|u|^2u = u\overline{u}u$ as a convolution, we obtain
	$$
	\widehat{|u|^2u}(\xi,\tau)=\int_{\xi_1,\xi_2}\int_{\tau_1,\tau_2}\hat{u}(\xi_1,\tau_1)\overline{\hat{u}(\xi_2,\tau_2)}\hat{u}(\xi-\xi_1+\xi_2,\tau-\tau_1+\tau_2).
	$$
	Hence
	$$
	\Vert |u|^2u \Vert^2_{X^{s+a,-b}}=\left\Vert  \int_{\xi_1,\xi_2}\int_{\tau_1,\tau_2}\dfrac{\langle\xi\rangle^{s+a}\hat{u}(\xi_1,\tau_1)\overline{\hat{u}(\xi_2,\tau_2)}\hat{u}(\xi-\xi_1+\xi_2,\tau-\tau_1+\tau_2)}{\langle\tau-|\xi|^{\alpha}\rangle^b}  \right\Vert^2_{L^2_\xi,L^2_\tau}.
	$$
	We define
	$$
	f(\xi,\tau)=|\hat{u}(\xi,\tau)|\langle \xi \rangle^s\langle\tau-|\xi|^\alpha\rangle^b
	$$
	and 
	$$
	M(\xi,\xi_1,\xi_2,\tau,\tau_1,\tau_2):=\dfrac{\langle\xi\rangle^{s+a}\langle\xi_1\rangle^{-s}\langle\xi_2\rangle^{-s}\langle\xi-\xi_1+\xi_2\rangle^{-s}}
	{\langle\tau-|\xi|^\alpha\rangle^{b}\langle\tau_1-|\xi_1|^\alpha\rangle^{b}\langle\tau_2-|\xi_2|^\alpha\rangle^{b}\langle\tau-\tau_1+\tau_2-|\xi-\xi_1+\xi_2|^\alpha\rangle^{b}}.
	$$
	On the other hand, observe that
	$$
	\begin{aligned}
	M(\xi,\xi_1,\xi_2,&\tau,\tau_1,\tau_2)f(\xi_1,\tau_1)f(\xi_2,\tau_2)f(\xi-\xi_1+\xi_2,\tau-\tau_1+\tau_2)=
	\\[7pt]
	&\dfrac{\langle\xi\rangle^{s+a}|\hat{u}(\xi_1,\tau_1)||\hat{u}(\xi_2,\tau_2)||\hat{u}(\xi-\xi_1+\xi_2,\tau-\tau_1+\tau_2)|}{\langle\tau-|\xi|^{\alpha }\rangle^b}
	\end{aligned}
	$$
	Then it is suffices to show that
	$$
	\left\Vert \int_{\xi_1,\xi_2}\int_{\tau_1,\tau_2}M(\xi,\xi_1,\xi_2,\tau,\tau_1,\tau_2)f(\xi_1,\tau_1)f(\xi_2,\tau_2)f(\xi-\xi_1+\xi_2,\tau-\tau_1+\tau_2)\right\Vert^2_{L^2_\xi,L^2_\tau}\lesssim \Vert u \Vert^6_{X^{s,b}}
	$$
	By applying Cauchy-Schwartz in the $\xi_1,\xi_2, \tau_1, \tau_2$ integral and then using Holder's inequality,we bound the norm above by
	$$
	\sup_{\xi,\tau}\left( \int_{\xi_1,\xi_2}\int_{\tau_1,\tau_2} M^2\right)\cdot \Vert f^2\ast f^2\ast f^2\Vert_{L^1_\xi,L^1_\tau}
	$$
	Using Young's inequality, the norm $\Vert f^2\ast f^2\ast f^2\Vert_{L^1_\xi,L^1_\tau}$ can be estimated by $\Vert f \Vert^6_{L^2_\xi,L^2_\tau}$ which is exactly 
	$\Vert u\Vert^6_{X^{s,b}}$. Therefore it is sufficient to show that the supremum above is finite. Using Lemma \ref{Lemma1} in the $ \tau_1, \tau_2$
	integrals, the supremum is bounded by
	$$
	\sup_{\xi,\tau}\int \dfrac{\langle\xi\rangle^{2s+2a}\langle\xi_1\rangle^{-2s}\langle\xi_2\rangle^{-2s}\langle\xi-\xi_1+\xi_2\rangle^{-2s}}{\langle\tau-|\xi|^\alpha\rangle^{2b}\langle\tau-|\xi_1|^\alpha+|\xi_2|^\alpha-|\xi-\xi_1+\xi_2|^\alpha\rangle^{6b-2}}d\xi_1 d\xi_2
	$$
	Using the relation $\langle\tau-a\rangle\langle\tau-b\rangle\gtrsim\langle a-b\rangle$, and due to $1/2-\epsilon<b<1/2$, the above reduces to
	$$
	\sup_{\xi}\int \dfrac{\langle\xi\rangle^{2s+2a}\langle\xi_1\rangle^{-2s}\langle\xi_2\rangle^{-2s}\langle\xi-\xi_1+\xi_2\rangle^{-2s}}{\langle|\xi|^\alpha-|\xi_1|^\alpha+|\xi_2|^\alpha-|\xi-\xi_1+\xi_2|^\alpha\rangle^{1-6\epsilon}}d\xi_1 d\xi_2
	$$
	We break the integral into two pieces.
	$$
	\big\{ |\xi_1-\xi|\geq 1\,\cap \,|\xi_1-\xi_2|\geq 1 \big\}\mbox{ and } \big\{|\xi_1-\xi|<1\, \cup\, |\xi_1-\xi_2|<1 \big\}.
	$$
	Then using  Lemma \ref{Lemma4} and Lemma \ref{Lemma5} we conclude the proof.
\end{proof}



\section{Proof of Theorem \ref{theorem1}}

\medskip
Fix $\alpha\in (\frac43,2)$ and $s\in (\frac{2-\alpha}{4}, \frac{\alpha-1}{2})$.
We pick an extension $\tilde{u}_0\in H^s(\mathbb{R})$ of $u_0$ such that $$\|\tilde{u}_0\|_{H^s(\mathbb{R})}\leq 2\|u_0\|_{H^s(\mathbb{R}^+)}.$$ Let $b=b(s)<\frac{1}{2}$   such that the estimates given in Proposition \ref{trilinear} are valid. 

\medskip
Let $Z^{s,b}$ the Banach space given by
$$
Z^{s,b}=C(\R_t,H^s(\R_x))\cap C(\R_x,H^{\frac{2s-1+\alpha}{2\alpha}}(\R_t))\cap X^{s,b}
$$
under the norm
\[\norm{v}_{Z^{s,b}} = \sup_{t \in \R} \norm{v(t,\cdot)}_{H^s} + \sup_{x \in \R} \norm{v(\cdot,x)}_{H^{\frac{2s-1+\alpha}{2\alpha}}} + \norm{v}_{X^{s,b} },\]
Let
\begin{equation}\label{operator}
\Lambda(u)(t)=\psi(t) e^{i t(-\Delta)^{\alpha/2}} \tilde{u}_{0}+\psi(t) \mathcal{D}\left(\psi_T|u|^{2}u\right)(t)+\psi(t) \mathcal{L} h(t)
\end{equation}

\medskip
where $$h(t)=\left.\left( \psi(t) f(t)-\psi(t) e^{i t(-\Delta)^{\alpha/2}}\left.\tilde{u}_{0}\right|_{x=0}-\psi(t) \mathcal{D}\left(\psi_T|u|^{2}u\right)\left.(t)\right|_{x=0}\right)\right|_{(0,+\infty)}$$

Our goal is to show that $\Lambda$ defines a contraction map on any ball of $Z^{s,b}$.

\medskip
Using Lemma \ref{grupo} we obtain
$$
\left\|\psi(t) e^{i t(-\Delta)^{\alpha/2} } \tilde{u}_{0}\right\|_{Z^{s,b}} \leq c\left\|\tilde{u}_{0}\right\|_{H^{s}(\mathbb{R})} \leq c\left\|u_{0}\right\|_{H^{s}\left(\mathbb{R}^{+}\right)}
$$

Note that the estimates in Lemma \ref{boundaryestimate} are valid for $-\frac12< s<\frac{\alpha-1}{2}$, in particular as $\frac43<\alpha<2$, these estimates are true for $\frac{2-\alpha}{4}< s<\frac{\alpha-1}{2}$ and imply
$$
\left\|\psi(t) \mathcal{L} h(t)\right\|_{Z^{s,b}} \leq c\|h\|_{H_{0}^{\frac{2s-1+\alpha}{2\alpha}}\left(\mathbb{R}^{+}\right)}.
$$

Moreover Lemmas \ref{sobolevh0}, \ref{grupo}, \ref{duhamel} and  \ref{gvt2} and Proposition \ref{trilinear} imply
$$
\begin{aligned}\|&h\|_{H^{\frac{2s-1+\alpha}{2\alpha}}\left(\mathbb{R}^{+}\right)}\\
 & \leq\left\|\chi_{(0,+\infty)}\left(\psi(t) f(t)-\psi(t) e^{i t(-\Delta)^{\alpha/2}}\tilde{u}_{0}-\psi(t) \mathcal{D}\left(\psi_T|u|^{2}u\right)(t)\right) \bigg|_{x=0}\right\|_{H^{\frac{2s-1+\alpha}{2\alpha}}(\mathbb{R})} \\[5pt]
&\leq c\left(\|f\|_{H^{\frac{2s-1+\alpha}{2\alpha}}(\mathbb{R}^+)}+\left\|\tilde{u}_{0}\right\|_{H^{s}(\mathbb{R})}+T^{\epsilon}\|u\|_{X^{s, b}}^{3}\right), 
  \end{aligned}
$$
for $\epsilon$ adequately small.

Consequently, we obtain
$$
\left\|\psi(t) \mathcal{L} h(t)\right\|_{Z^{s,b}} \leq c\left(\|f\|_{H^{\frac{2s-1+\alpha}{2\alpha}}(\mathbb{R}^+)}+\left\|u_{0}\right\|_{H^{s}(\mathbb{R}^+)}+T^{\epsilon}\|u\|_{Z^{s, b}}^{3}\right).
$$
On the other hand, from Lemma \ref{duhamel} and Proposition \ref{trilinear} we arrive
$$
\left\|\psi(t) \mathcal{D}\left(\psi_T|u|^{2}u\right)(t)\right\|_{Z^{s,b}} \leq cT^{\epsilon }\|u\|_{X^{s, b}}^{3}\leq cT^{\epsilon}\|u\|_{Z^{s, b}}^{3}.
$$
Thus, we get
$$
\left\|\Lambda u\right\|_{Z^{s,b}} \leq c\left(\|f\|_{H^{\frac{2s-1+\alpha}{2\alpha}}(\mathbb{R}^+)}+\left\|u_{0}\right\|_{H^{s}(\mathbb{R}^+)}+T^{\epsilon}\|u\|_{Z^{s, b}}^{3}\right).
$$
Similarly,
$$
\|\Lambda u_1 - \Lambda u_2\|_{Z^{s,b}}\leq  C_2T^{\epsilon}(\norm{u_1}_{Z^{s,b}}^2+\norm{u_2}_{Z^{s,b}}^2)\norm{u_1-u_2}_{Z^{s,b}},
$$
for $u_1(0,x) = u_2(0,x)$.

\medskip

Consider in $Z_{s,b}$, the ball defined by 
$
B=\{ u\in Z^{s,b}; \|u\|_{Z^{s,b}}\leq M\},
$
where $$M=2 c\left(\|u_0\|_{H^s(\R^+)}+\|f\|_{H^{\frac{2s-1+\alpha}{2\alpha}}(\R^+)}\right).$$ Lastly, choosing $T=T(M)$ sufficiently small, such that 
\[\|\Lambda u\|_{Z^{s,b}}\leq M\ \text{and}\ \|\Lambda u_1 - \Lambda u_2\|_{Z^{s,b}}\leq \frac12\norm{u_1-u_2}_{Z^{s,b}}, \] it follows that $\Lambda$ is a contraction map on $B$, and it completes the proof of Theorem \ref{theorem1}.

Finally, we prove \eqref{regularity}. By \eqref{operator} we have for $t\in [0,T]$ 
$$u(x,t)-L(u_0,f)=-\psi(t) \mathcal{D}\left(\psi_{T}|u|^{2} u\right)(t)-\psi(t) \mathcal{L}( \mathcal{D}\left(\psi_{T}|u|^{2} u\right).(0,t)\bigg|_{x=0}).$$

Therefore, by using Lemmas \ref{duhamel}, \ref{boundaryestimate} and \ref{gvt2} and Proposition \ref{trilinear} we get, for $\epsilon$ adequately small, $c>-\frac12$ and  $a<\min\{\frac{\alpha-1}{2},\frac{4s+\alpha-2}{2}\}$
\begin{equation}
\begin{split}
\left\|u-L_{u_0,f}\right\|_{C(\R_t: H^{s+a}(\R))} &\lesssim T^{\epsilon}\left\||u|^{2} u\right\|_{X^{s+a,c}}\lesssim T^{\epsilon}\|u\|_{X^{s,b}}^3\\
&\lesssim \left(\|u_0\|_{H^s(\R^+)}+\|f\|_{H^{\frac{2s-1+\alpha}{2\alpha}}(\R^+)}\right)^3,
\end{split}
\end{equation}
where in the last inequality we used that the solution is contained in the ball $ B$.

\section{Proof of Corollary \ref{theorem2}}
	The proof is a direct application of the Fourier transform restriction method joint with the trilinear estimate obtained in this work. We only give the principal ideas of the proof.  Fix  $b>\frac12$ and  $\frac{2-\alpha}{4}<s<\frac12$. Define the classical truncated operator
\begin{equation}\Lambda(u)(t)=\psi(t) e^{i t(-\Delta)^{\alpha / 2}} u_{0}+\psi(t) \mathcal{D}\left(\psi_{T}|u|^{2} u\right)(t),
\end{equation}
on the space $X^{s,b}$, such that.

By using the same ideas of the proof of Theorem \ref{theorem1} there is a fix point $u(x,t)$ on the ball $B=\left\{u \in X^{s, b} ;\|u\|_{X^{s, b}} \leq c \|u_0\|_{H^s(\R^n)}\right\}$. Now, by using the immersion $X^{s,b}\subset C(\R:H^s(\R^+))$, Lemma \ref{duhamel} and Proposition \ref{trilinear} we have that for $a<2\alpha-1$
\begin{equation}
\begin{split}
\|u(t)-L_{u_0,f}\|_{C(\R:H^{s+a}(\R))}&= \|\psi(t) \mathcal{D}\left(\psi_{T}|u|^{2} u\right)(t)\|_{C(\R:H^{s+a}(\R))}\\
&\lesssim \|\psi(t) \mathcal{D}\left(\psi_{T}|u|^{2} u\right)(t)\|_{X^{s+a,b-1}}\\
&\lesssim T^{\epsilon}\|u\|_{X^{s, b}}^{3}\lesssim \|u_0\|_{H^s(\R)}^3,
\end{split}
\end{equation}
for $\epsilon$ adequately small.
This proves the result.
 \section*{Acknowledgements}
The second author wishes to thank the Federal University of Alagoas during his postdoctoral stay, where the paper was
written, and Capes for the financial support.

\medskip

\appendix

In appendices, we present the proof of the estimate of the classical Duhamel operator and the more technical integral estimates.

\section{Proof of Lemma \ref{duhamel}}\label{apduhamel}

We only prove (b), since the assertions (a) and (c) are standard and can be found in \cite{GVT}.  To do this we take a function $\theta(\tau)\in C^{\infty}(\mathbb{R})$ such that $\theta(\tau)=1$ for $|\tau|<\frac{1}{2}$ and supp $ \theta \subset[-\frac{2}{3},\frac{2}{3}]$, then
\begin{align*}
\mathcal{F}_x&\left( \psi(t)\int_0^te^{i(t-t')(-\Delta)^{\alpha/2}}w(x,t')\right)(\xi)\\
&=\psi(t)\int_{\tau}\frac{e^{it\tau}-e^{it|\xi|^{\alpha}}}{\tau-|\xi|^{\alpha}}\hat{w}(\xi,\tau)d\tau
\\
&\quad= \psi(t)e^{it\xi^2}\int_{\tau}\frac{e^{it(\tau-|\xi|^{\alpha})}-1}{\tau-|\xi|^{\alpha}}\theta(\tau-|\xi|^{\alpha})\hat{w}(\xi,\tau)d\tau\\
& \quad\quad+\psi(t)\int_{\tau}e^{it\tau}\frac{1-\theta(\tau-|\xi|^{\alpha})}{\tau-|\xi|^{\alpha}}\hat{w}(\xi,\tau)d\tau\\
&\quad \quad -\psi(t)e^{it|\xi|^{\alpha}}\int_{\tau}\frac{1-\theta(\tau-|\xi|^{\alpha})}{\tau-|\xi|^{\alpha}}\hat{w}(\xi,\tau)d\tau\\
&:=\mathcal{F}_xw_1+\mathcal{F}_xw_2-\mathcal{F}_xw_3.
\end{align*}
By the power series expansion for the function $e^{it(\tau-|\xi|^{\alpha})}$, we have that $$w_1(x,t)=\displaystyle\sum_{k=1}^{\infty}\frac{\psi_k(t)}{k!}e^{it(-\Delta)^{\alpha/2}}\phi_k(x),$$ 
where $\psi_k(t)=i^kt^k\theta(t)$ and $\hat{\phi}_k(\xi)=\int_{\tau}(\tau-|\xi|^{\alpha})^{k-1}\theta(\tau-|\xi|^{\alpha})\hat{w}(\xi,\tau)d\tau.$ Using Lemma \ref{grupo} (b), it suffices to show that $\|\phi_k\|_{H^s(\mathbb{R})}\leq c\|u\|_{X^{s,c}}$.
Using the definition of $\phi_k$ and the Cauchy-Schwarz inequality we get
\begin{eqnarray*}
	\|\phi_k\|_{H^s(\mathbb{R}_x)}^2&=&c\int_{\xi}\langle \xi \rangle^{2s}\left( \int_{\{\tau:|\tau-|\xi|^{\alpha}|\leq\frac{2}{3}\}}\sum_{k=1}^{\infty}(\tau-|\xi|^{\alpha})^{k-1}\theta(\tau-|\xi|^{\alpha})\hat{u}(\xi,\tau)\right)^2d\xi\\
	&\leq&c\int_{\xi}\langle \xi \rangle^{2s}\int_{\tau}\langle \tau-|\xi|^{\alpha}\rangle^{2c}|\hat{u}(\xi,\tau)|^2d\tau d\xi.
\end{eqnarray*}
This completes the estimate for $w_1$. To treat $w_2$, we use Lemma \ref{sobolevh0} part (iv) (to remove the cuttof function $\psi(t)$) and change variables $\eta=|\xi|^{\alpha}$ to get
\begin{equation}\label{28101}
\|w_2\|_{C\big(\mathbb{R}_x;\,H^{\frac{2s-1+\alpha}{2\alpha}}(\mathbb{R}_t)\big)}\leq c\int_{\tau}\langle \tau\rangle^{\frac{2s-1+\alpha}{\alpha}}\left( \int_{\eta=0}^{+\infty}\langle \tau-\eta\rangle^{-1}\eta^{\frac{1-\alpha}{\alpha}}\big|\hat{w}(\eta^{\frac{1}{\alpha}},\tau)\big|d\eta\right)^2d\tau.
\end{equation}
By Cauchy-Schwarz the right had side of \eqref{28101} is bounded by 
\begin{equation*}
c\int_{\tau}\langle \tau\rangle^{\frac{2s-1+\alpha}{\alpha}}G(\tau)\int_{\eta=0}^{\infty}\langle \eta \rangle^{s}\langle \tau-\eta\rangle^{2c}|\eta|^{\frac{1-\alpha}{\alpha}}\big|\hat{w}(\pm\eta^{\frac{1}{\alpha}},\tau)\big|^2d\eta d\tau,
\end{equation*}
where $$G(\tau)=\int_{\eta}\langle \tau-\eta\rangle^{-2-2c}\langle \eta \rangle^{-s}|\eta|^{\frac{1-\alpha}{\alpha}} d\eta.$$ 

Then, we need to prove $G(\tau)\leq c \langle\tau\rangle^{\frac{1-2s-\alpha}{\alpha}}$. To do this, we estimate we estimate separately on the cases $|\eta|\leq1$, $2|\eta| \leq |\tau|$ and $|\tau|>2\eta$.

For $|\eta|<1$ we have that
$
G(\tau)\leq \int_{|\eta|\leq1}|\eta|^{\frac{1-\alpha}{\alpha}}d\eta,
$
that  is finite for any positive $\alpha$.

Now assume $2\leq2|\eta|\leq |\tau|$. In this case we use $\langle \tau-\eta \rangle \sim \langle \tau \rangle$ to obtain
\begin{equation}\label{aux1}
\begin{split}
G(\tau)&\leq  \langle \tau\rangle^{-2-2c} \int_{|\tau|>\eta} \langle \eta \rangle^{-s+\frac{1-\alpha}{\alpha}}d\tau\leq \langle \tau\rangle^{-2-2c}\langle\tau\rangle^{-s+\frac{1-\alpha}{\alpha}+1},
\end{split}
\end{equation}
when we have used that $s>-\frac{1}{\alpha}$. Then note that the right hand side of \eqref{aux1} is controlled by c $\langle\tau\rangle^{\frac{1-2s-\alpha}{\alpha}}$ if $(2-\alpha )s<\alpha+2c\alpha-1$. 
This completes the estimate for $w_2$.

Finally, we estimate $w_3$. To do this we write $w_3=\psi(t)e^{it(-\Delta)^{\frac{\alpha}{2}}}\phi(x)$, where $\hat{\phi}(\xi)=\int\frac{1-\theta(\tau-|\xi|^{\alpha})}{\tau-|\xi|^{\alpha}}\hat{w}(\xi,\tau)d\tau$. Using Lemma \ref{grupo} (b) and the Cauchy-Schwarz inequality, we obtain
\begin{eqnarray*}
	\|w_3\|_{C(\mathbb{R}_x;H^{\frac{2s-1+\alpha}{\alpha}}(\mathbb{R}_t))}&=&c\|\psi(t)e^{it(-\Delta)^{\frac{\alpha}{2}}}\phi(x)\|_{C(\mathbb{R}_x;H^{\frac{2s-1+\alpha}{2\alpha}}(\mathbb{R}_t))}\leq c\|\phi\|_{H^s(\mathbb{R})}\\
	&\leq&c\int_{\xi}\langle \xi \rangle^{2s}\left(\int_{\tau}|\hat{w}(\xi,\tau)|^2\langle\tau-|\xi|^{\alpha}\rangle^{2c}d\tau \int \frac{d\tau}{\langle\tau-|\xi|^{\alpha}\rangle^{2-2c}}\right) d\xi.
\end{eqnarray*}Since $c>-\frac{1}{2}$, we have $\int \frac{1}{\langle\tau-|\xi|^{\alpha}\rangle^{2-2c}}d\tau\leq c$. This completes the proof of Lemma \ref{duhamel}.

\section{Proof of Lemma \ref{Lemma4}}\label{proof1}
First, from  Lemma \ref{Lemma3} we have
$$
\sup_{\xi}\int_{A}\; \dfrac{\langle\xi\rangle^{2s+2a}\langle\xi_1\rangle^{-2s}\langle\xi_2\rangle^{-2s}\langle\xi-\xi_1+\xi_2\rangle^{-2s}}{\left\langle \dfrac{|\xi_1-\xi||\xi_1-\xi_2|}{(|\xi_1-\xi_2|+|\xi_1-\xi|+|\xi|)^{2-\alpha}}\right\rangle^{1-6\epsilon}}d\xi_1 d\xi_2.
$$

On the other hand, since $\left\{|\xi_1-\xi|<1\, or\, |\xi_1-\xi_2|<1\right\}$ we have
$$
\langle \xi_1\rangle\langle\xi-\xi_1+\xi_2\rangle\backsim\langle\xi\rangle\langle\xi_2\rangle
$$
Therefore, we have
\begin{equation}
\sup_{\xi}\int_{A}\;  \dfrac{\langle\xi\rangle^{2a}\langle\xi_2\rangle^{-4s}}{\left\langle \dfrac{|\xi_1-\xi||\xi_1-\xi_2|}{(|\xi_1-\xi_2|+|\xi_1-\xi|+|\xi|)^{2-\alpha}}\right\rangle^{1-6\epsilon}}d\xi_1 d\xi_2.\label{eq1Lemma4}
\end{equation}
We break the integral into two cases $\{|\xi_1|>1\}$ and $\{|\xi_1|\leq1\}$.

\medskip
\textbf{Case I :} $\{|\xi_1|>1\}$, then from \eqref{eq1Lemma4}, we have
\begin{equation}
\sup_{\xi}\int_{A}\;  \dfrac{\langle\xi\rangle^{2a}\langle\xi_2\rangle^{-4s}{(|\xi_1-\xi_2|+|\xi_1-\xi|+|\xi|)^{(2-\alpha)(1-6\epsilon)}}}
{\left\langle|\xi_1-\xi||\xi_1-\xi_2|\right\rangle^{1-6\epsilon}}d\xi_1 d\xi_2.\label{eq2Lemma4}
\end{equation}
Now we divide into two case $\{|\xi_2|<|\xi|\}$ and $\{|\xi_2|\geq|\xi| \}$.

\medskip
\textbf{Case I.1:} $\{|\xi_2|<|\xi|\}$, first observe that
$$
|\xi_1-\xi_2|+|\xi_1-\xi|+|\xi|\lesssim \langle\xi\rangle,
$$
thus from \eqref{eq2Lemma4}, we get
$$
\sup_{\xi}\int_{A}\; \dfrac{\langle\xi\rangle^{2a+(2-\alpha)(1-6\epsilon)}\langle\xi_2\rangle^{-4s}}
{\left\langle|\xi_1-\xi||\xi_1-\xi_2|\right\rangle^{1-6\epsilon}}d\xi_1 d\xi_2.
$$
we use the substitution $x=(\xi_1-\xi)(\xi_1-\xi_2)$ in the $\xi_1$ integral. Therefore, the integral above is bounded by
$$
\sup_{\xi}\int_{A}\;  \dfrac{\langle\xi\rangle^{2a+(2-\alpha)(1-6\epsilon)}\langle\xi_2\rangle^{-4s}}
{\left\langle x\right\rangle^{1-6\epsilon}\sqrt{|4x+(\xi-\xi_2)^2|}}dx\,d\xi_2.
$$
Using Lemma \ref{Lemma2} and then Lemma \ref{Lemma1} again, we bound the supremum of the integral above by
$$
\sup_{\xi}\int_{A}\;  \dfrac{\langle\xi\rangle^{2a+(2-\alpha)(1-6\epsilon)}}
{\langle \xi_2-\xi\rangle^{1-12\epsilon}\langle\xi_2\rangle^{4s}}dx\,d\xi_2\lesssim
\begin{cases}
\langle\xi\rangle^{2a+1-\alpha+6\alpha\epsilon}&s\geq1/4,
\\[5pt]
\langle\xi\rangle^{2a-4s+2-\alpha+6\alpha\epsilon}&s<1/4.
\end{cases}
$$
Which is finite for $a<\min\{\frac{\alpha-1}{2},\frac{4s+\alpha-2}{2}\}$.

\medskip
\textbf{Case I.2:} $\{|\xi_2|\geq|\xi| \}$, first observe that
$$
|\xi_1-\xi_2|+|\xi_1-\xi|+|\xi|\lesssim \langle\xi_2\rangle,
$$
thus from \eqref{eq2Lemma4}, we get
$$
\sup_{\xi}\int_{A}\;  \dfrac{\langle\xi\rangle^{2a}\langle\xi_2\rangle^{-4s+(2-\alpha)(1-6\epsilon)}}
{\left\langle|\xi_1-\xi||\xi_1-\xi_2|\right\rangle^{1-6\epsilon}};d\xi_1 d\xi_2.
$$
Then the estimate follows as the case I.1, consequently the integral above is bounded by
$$
\sup_{\xi}\int_{A}\;  \dfrac{\langle\xi\rangle^{2a}}
{\langle \xi_2-\xi\rangle^{1-12\epsilon}\langle\xi_2\rangle^{4s-(2-\alpha)(1-6\epsilon)}}dx\,d\xi_2\lesssim
\begin{cases}
\langle\xi\rangle^{2a-4s+2-\alpha+6\alpha\epsilon},&s<\frac{3-\alpha}{2},
\\[5pt]
\langle\xi\rangle^{2a-1+12\epsilon},&s\geq\frac{3-\alpha}{2}.
\end{cases}
$$
Finally, the case $\{|\xi_1|\leq1\}$ follows as the case I. Therefore the term \eqref{priceqlemma4} is bounded.

\section{Proof of Lemma \ref{Lemma5}}\label{proof2}

	First from  Lemma \ref{Lemma3} we obtain that \eqref{priceqlemma5} is bounded by
	\begin{equation}
	\sup_{\xi}\int_{B}\; \dfrac{\langle\xi\rangle^{2s+2a}\langle\xi_1\rangle^{-2s}\langle\xi_2\rangle^{-2s}\langle\xi-\xi_1+\xi_2\rangle^{-2s}}{\max\left\lbrace 1,\frac{|\xi_1-\xi||\xi_1-\xi_2|}{(|\xi_1-\xi|+|\xi_1-\xi_2|+|\xi|)^{2-\alpha}}\right\rbrace^{1-6\epsilon}}d\xi_1 d\xi_2.\label{eq1lemma5}
	\end{equation}
	
	Now we divide the integral into two case, the first set contains the term with $ |\xi_1-\xi||\xi_1-\xi_2|\ll(|\xi_1-\xi|+|\xi_1-\xi_2|+|\xi|)^{2-\alpha}$ and the second set contains the remaining terms.
	
	\medskip
	\textbf{Case I:} $|\xi_1-\xi||\xi_1-\xi_2|\ll(|\xi_1-\xi|+|\xi_1-\xi_2|+|\xi|)^{2-\alpha}$. First, observe that  \eqref{eq1lemma5} is bounded by 
	\begin{equation}
	\sup_{\xi}\int_{B\cap D}\; \dfrac{\langle\xi\rangle^{2s+2a}}{\langle\xi_1\rangle^{2s}\langle\xi_2\rangle^{2s}\langle\xi-\xi_1+\xi_2\rangle^{2s}}d\xi_1 d\xi_2,\label{eq2lemma5}
	\end{equation}
	where $D:=\left\lbrace |\xi_1-\xi||\xi_1-\xi_2|\ll(|\xi_1-\xi|+|\xi_1-\xi_2|+|\xi|)^{2-\alpha}\right\rbrace$. 
	
	\medskip
	On the other hand, since $1<\alpha<2$, $|\xi_1-\xi|\geq1$ and $|\xi_1-\xi_2|\geq1$ we have
	$$
	|\xi_1-\xi||\xi_1-\xi_2|\ll(|\xi_1-\xi|+|\xi_1-\xi_2|+|\xi|)^{2-\alpha}\lesssim|\xi_1-\xi||\xi_1-\xi_2|+|\xi|^{2-\alpha} ,
	$$
	consequently $|\xi_1-\xi||\xi_1-\xi_2|\lesssim |\xi|^{2-\alpha}$, thus \eqref{eq2lemma5} is bounded by
	\begin{equation}
	\sup_{\xi}\int_{B\cap E}\; \dfrac{\langle\xi\rangle^{2s+2a}}{\langle\xi_1\rangle^{2s}\langle\xi_2\rangle^{2s}\langle\xi-\xi_1+\xi_2\rangle^{2s}}d\xi_1 d\xi_2,\nonumber
	\end{equation}
	where $E:=\left\lbrace|\xi_1-\xi||\xi_1-\xi_2|\lesssim |\xi|^{2-\alpha}\right\rbrace$. 
	
	\medskip
	On the other hand, observe that for all $M\in\R^+$, we have
	\begin{equation}
	\sup_{|\xi|\leq M}\int_{B\cap E}\; \dfrac{\langle\xi\rangle^{2s+2a}}{\langle\xi_1\rangle^{2s}\langle\xi_2\rangle^{2s}\langle\xi-\xi_1+\xi_2\rangle^{2s}}d\xi_1 d\xi_2,\nonumber
	\end{equation}
	is bounded. In particular we choose $M\in\R^+$ big enough, then it is sufficient to prove that
	\begin{equation}
	\sup_{|\xi|>M}\int_{B\cap E}\; \dfrac{\langle\xi\rangle^{2s+2a}}{\langle\xi_1\rangle^{2s}\langle\xi_2\rangle^{2s}\langle\xi-\xi_1+\xi_2\rangle^{2s}}d\xi_1 d\xi_2,\label{eq3lemma5}
	\end{equation}
	is finite. 
	
	\medskip
	To prove \eqref{eq3lemma5}, first observe that, the condition on the integral together with $|\xi|>M$  implies,  that $|\xi_1-\xi|\ll|\xi| $ and $|\xi_1-\xi_2|\ll|\xi| $. Therefore \eqref{eq3lemma5} is bounded by
	\begin{equation}
	\sup_{|\xi|>M}\int_{B\cap E}\; \dfrac{\langle\xi\rangle^{2s+2a}}{\langle\xi\rangle^{2s}\langle\xi\rangle^{2s}\langle\xi\rangle^{2s}}d\xi_1 d\xi_2\lesssim\sup_{|\xi|>M}\langle\xi\rangle^{2a-4s}\int_{B\cap E}\; \dfrac{\langle\xi\rangle^{(2-\alpha)(1+\epsilon)}}{\langle\xi_1-\xi\rangle^{1+\epsilon}\langle\xi_1-\xi_2\rangle^{1+\epsilon}}d\xi_1 d\xi_2,\nonumber
	\end{equation}
	in the last bound we used $|\xi_1-\xi||\xi_1-\xi_2|\lesssim |\xi|^{2-\alpha}$. Then using  Lemma \ref{Lemma1} we obtain
	\begin{equation}
	\sup_{|\xi|>M}\langle\xi\rangle^{2a-4s}\int_{B\cap E}\; \dfrac{\langle\xi\rangle^{(2-\alpha)(1+\epsilon)}}{\langle\xi_1-\xi\rangle^{1+\epsilon}\langle\xi_1-\xi_2\rangle^{1+\epsilon}}d\xi_1 d\xi_2\lesssim\sup_{|\xi|>M}\langle\xi\rangle^{2a-4s+(2-\alpha)(1+\epsilon)},\nonumber
	\end{equation}
	which is finite for $a<\frac{\alpha-2+4s}{2}$
	
	\medskip
	\textbf{Case II:} The set is the complement of $|\xi_1-\xi||\xi_1-\xi_2|\ll(|\xi_1-\xi|+|\xi_1-\xi_2|+|\xi|)^{2-\alpha}$. First, observe that \eqref{eq1lemma5} is bounded by
	\begin{equation}
	\sup_{\xi}\int_{B\cap F}\; \dfrac{\langle\xi\rangle^{2s+2a}{(|\xi_1-\xi|+|\xi_1-\xi_2|+|\xi|)^{(2-\alpha)(1-6\epsilon)}}}{\langle\xi_1\rangle^{2s}\langle\xi_2\rangle^{2s}\langle\xi-\xi_1+\xi_2\rangle^{2s}\left\lbrace|\xi_1-\xi||\xi_1-\xi_2|\right\rbrace^{1-6\epsilon} }\;d\xi_1 d\xi_2,\label{eq4lemma5}
	\end{equation}
	where $F:=\left\lbrace|\xi_1-\xi||\xi_1-\xi_2|\ll(|\xi_1-\xi|+|\xi_1-\xi_2|+|\xi|)^{2-\alpha}\right\rbrace^{c}$.
	
	On the other hand observe that $F\subseteq \left\lbrace|\xi_1-\xi||\xi_1-\xi_2|\gtrsim|\xi|^{2-\alpha} \right\rbrace$, consequently \eqref{eq4lemma5} is bounded by
	\begin{equation}
	\sup_{\xi}\int_{B\cap G}\; \dfrac{\langle\xi\rangle^{2s+2a}{(|\xi_1-\xi|+|\xi_1-\xi_2|+|\xi|)^{(2-\alpha)(1-6\epsilon)}}}{\langle\xi_1\rangle^{2s}\langle\xi_2\rangle^{2s}\left\lbrace|\xi_1-\xi||\xi_1-\xi_2|\right\rbrace^{1-6\epsilon}}\;d\xi_1 d\xi_2,\label{eq5lemma5}
	\end{equation}
	where $G:=\left\lbrace|\xi_1-\xi||\xi_1-\xi_2|\gtrsim|\xi|^{2-\alpha} \right\rbrace$. Then 
	by symmetry, it is enough to consider $|\xi_1-\xi_2|\geq|\xi_1-\xi|$.
	
	\medskip
	On the other hand, we define $H:=B\cap G\cap \left\lbrace|\xi_1-\xi_2|\geq|\xi_1-\xi|\right\rbrace$ therefore from \eqref{eq5lemma5}, it is enough to study
	\begin{equation}
	\sup_{\xi}\int_{H}\; \dfrac{\langle\xi\rangle^{2s+2a}{(|\xi_1-\xi_2|+|\xi|)^{(2-\alpha)(1-6\epsilon)}}}{\langle\xi_1\rangle^{2s}\langle\xi_2\rangle^{2s}\langle\xi-\xi_1+\xi_2\rangle^{2s}\left\lbrace|\xi_1-\xi||\xi_1-\xi_2|\right\rbrace^{1-6\epsilon} }\;d\xi_1 d\xi_2,\label{eq6lemma5}
	\end{equation}
	
	Since $|\xi_1-\xi|\geq1$, $|\xi_1-\xi_2|\geq1$ we have $|\xi_1-\xi||\xi_1-\xi_2|\backsim\langle\xi_1-\xi\rangle\langle\xi_1-\xi_2\rangle$, thus \eqref{eq6lemma5} is equivalent to
	\begin{equation}
	\sup_{\xi}\int_{H}\; \dfrac{\langle\xi\rangle^{2s+2a}{(|\xi_1-\xi_2|+|\xi|)^{(2-\alpha)(1-6\epsilon)}}}{\langle\xi_1\rangle^{2s}\langle\xi_2\rangle^{2s}\langle\xi-\xi_1+\xi_2\rangle^{2s}\left\lbrace\langle\xi_1-\xi\rangle\langle\xi_1-\xi_2\rangle\right\rbrace^{1-6\epsilon} }\;d\xi_1 d\xi_2,\label{eq7lemma5}
	\end{equation}
	
	Now to estimate \eqref{eq7lemma5}, we break the integral into  three  region
	$|\xi_1-\xi_2|\approx|\xi|$, $|\xi_1-\xi_2|\ll|\xi|$ and $|\xi|\ll|\xi_1-\xi_2|$.
	
	\medskip
	\textbf{Region 1:} $|\xi_1-\xi_2|\approx|\xi|$,
	\begin{equation}
	\sup_{\xi}\int_{H_1}\; \dfrac{\langle\xi\rangle^{2s+2a}{(|\xi_1-\xi_2|+|\xi|)^{(2-\alpha)(1-6\epsilon)}}}{\langle\xi_1\rangle^{2s}\langle\xi_2\rangle^{2s}\langle\xi-\xi_1+\xi_2\rangle^{2s}\left\lbrace\langle\xi_1-\xi\rangle\langle\xi_1-\xi_2\rangle\right\rbrace^{1-6\epsilon} }\;d\xi_1 d\xi_2,\label{eq8lemma5}
	\end{equation}
	where $H_1:=H\cap \left\lbrace|\xi_1-\xi_2|\approx|\xi|\right\rbrace$. Then \eqref{eq8lemma5} is equivalent to
	\begin{equation}
	\sup_{\xi}\int_{H_1}\; \dfrac{\langle\xi\rangle^{2s+2a+(2-\alpha)(1-6\epsilon)-(1-6\epsilon)}}{\langle\xi_1\rangle^{2s}\langle\xi_2\rangle^{2s}\langle\xi-\xi_1+\xi_2\rangle^{2s}\langle\xi_1-\xi\rangle^{1-6\epsilon}}\;d\xi_1 d\xi_2,\label{eq9lemma5}
	\end{equation}
	
	Now, we have to consider the cases $ \frac{2-\alpha}{4}<s\leq\frac{1}{4}$ and $ \frac{1}{4}<s<\frac{1}{2}$ separately. 
	
	\medskip
	\textbf{Case 1}: $ \frac{2-\alpha}{4}<s\leq\frac{1}{4}$. 
	
	First observe that \eqref{eq9lemma5} is bounded by
	\begin{equation}
	\sup_{\xi}\int_{\xi_1}\; \dfrac{\langle\xi\rangle^{2s+2a-(\alpha-1)(1-6\epsilon)}}{\langle\xi_1\rangle^{2s}\langle\xi_1-\xi\rangle^{1-6\epsilon}}\;\left\lbrace\int_{|\xi_1-\xi_2|\approx|\xi|}\dfrac{1}{\langle\xi_2\rangle^{2s}\langle\xi-\xi_1+\xi_2\rangle^{2s}} d\xi_2\right\rbrace\;d\xi_1.\label{eq10lemma5}
	\end{equation}
	
	On the other hand, using Lemma \ref{Lema0} we obtain that \eqref{eq10lemma5} is bounded by
	\begin{equation}
	\sup_{\xi}\int_{\xi_1}\; \dfrac{\langle\xi\rangle^{2s+2a-(\alpha-1)(1-6\epsilon)}\langle\xi\rangle^{1-4s}}{\langle\xi_1\rangle^{2s}\langle\xi_1-\xi\rangle^{1-6\epsilon}}\;d\xi_1=\sup_{\xi}\int_{\xi_1}\; \dfrac{\langle\xi\rangle^{-2s+2a-(\alpha-1)(1-6\epsilon)+1}}{\langle\xi_1\rangle^{2s}\langle\xi_1-\xi\rangle^{1-6\epsilon}}\;d\xi_1,\label{eq11lemma5}
	\end{equation}
	
	Finally applying Lemma \ref{Lemma1} we obtain that \eqref{eq11lemma5} is bounded by
	\begin{equation}
	\sup_{\xi}\int_{\xi_1}\; \dfrac{\langle\xi\rangle^{-2s+2a-(\alpha-1)(1-6\epsilon)+1}}{\langle\xi_1\rangle^{2s}\langle\xi_1-\xi\rangle^{1-6\epsilon}}\;d\xi_1\lesssim\sup_{\xi}\; \langle\xi\rangle^{-2s+2a-(\alpha-1)(1-6\epsilon)+1-2s+6\epsilon},
	\end{equation}
	
	which is finite for $a<\frac{4s+\alpha-2}{2}$.
	
	\medskip
	\textbf{Case 2}:  $ \frac{1}{4}<s<\frac{1}{2}$ 
	
	\medskip
	Here we use Lemma \ref{Lemma1} in the variable $\xi_2$, then \eqref{eq9lemma5} is bounded by
	\begin{equation}
	\begin{aligned}
	\sup_{\xi}\int_{\xi_1}\; \dfrac{\langle\xi\rangle^{2s+2a+(2-\alpha)(1-6\epsilon)-(1-6\epsilon)}}{\langle\xi_1\rangle^{2s}\langle\xi-\xi_1\rangle^{4s-1}\langle\xi_1-\xi\rangle^{1-6\epsilon}}\;d\xi_1& 
	=\sup_{\xi}\int_{\xi_1}\; \dfrac{\langle\xi\rangle^{2s+2a-(\alpha-1)(1-6\epsilon)}}{\langle\xi_1\rangle^{2s}\langle\xi-\xi_1\rangle^{4s-6\epsilon}}\;d\xi_1
	\\[5pt]
	&\lesssim\sup_{\xi}\langle\xi\rangle^{2s+2a-(\alpha-1)(1-6\epsilon)}\langle\xi\rangle^{-2s}
	\end{aligned}\nonumber
	\end{equation}
	where the last inequality we use again  Lemma \ref{Lemma1} in the variable $\xi_1$. Therefore it is finite for $a<\frac{\alpha-1}{2}$.
	
	\medskip
	\textbf{Region 2:} $|\xi_1-\xi_2|\ll|\xi|$.
	\begin{equation}
	\sup_{\xi}\int_{H_2}\; \dfrac{\langle\xi\rangle^{2s+2a}{(|\xi_1-\xi_2|+|\xi|)^{(2-\alpha)(1-6\epsilon)}}}{\langle\xi_1\rangle^{2s}\langle\xi_2\rangle^{2s}\langle\xi-\xi_1+\xi_2\rangle^{2s}\left\lbrace\langle\xi_1-\xi\rangle\langle\xi_1-\xi_2\rangle\right\rbrace^{1-6\epsilon} }\;d\xi_1 d\xi_2,\label{eq12lemma5}
	\end{equation}
	where $H_2:=H\cap \left\lbrace|\xi_1-\xi_2|\ll|\xi|\right\rbrace$. Then \eqref{eq12lemma5} is equivalent to
	\begin{equation}
	\sup_{\xi}\int_{H_2}\; \dfrac{\langle\xi\rangle^{2s+2a+(2-\alpha)(1-6\epsilon)}}{\langle\xi_1\rangle^{2s}\langle\xi_2\rangle^{2s}\langle\xi-\xi_1+\xi_2\rangle^{2s}\langle\xi_1-\xi\rangle^{1-6\epsilon}\langle\xi_1-\xi_2\rangle^{1-6\epsilon}}\;d\xi_1 d\xi_2,\label{eq13lemma5}
	\end{equation}
	On the other hand, since $|\xi_1-\xi_2|\ll|\xi|$, we have
	$$
	|\xi|\lesssim |\xi-\xi_1+\xi_2|
	,\;\;
	|\xi|\lesssim |\xi_2|,\;\;|\xi|\lesssim|\xi_1|.
	$$
	Thus \eqref{eq13lemma5} is bounded by
	\begin{equation}
	\sup_{\xi}\int_{H_2}\; \dfrac{\langle\xi\rangle^{2a-4s+(2-\alpha)(1-6\epsilon)}}{\langle\xi_1-\xi\rangle^{1-6\epsilon}\langle\xi_1-\xi_2\rangle^{1-6\epsilon}}\;d\xi_1\,d\xi_2\lesssim\sup_{\xi}\int_{H_2}\; \dfrac{\langle\xi\rangle^{2a-4s+(2-\alpha)(1-6\epsilon)+14\epsilon}}{\langle\xi_1-\xi\rangle^{1+\epsilon}\langle\xi_1-\xi_2\rangle^{1+\epsilon}}\;d\xi_1\,d\xi_2.\nonumber
	\end{equation}
	
	In the last estimative is due to $|\xi_1-\xi_2|\ll|\xi|$ and $|\xi_1-\xi|\ll|\xi|$. Then using  Lemma \ref{Lemma1} we obtain a new limitation by
	$$
	\sup_{\xi}\langle\xi\rangle^{2a-4s+2-\alpha+14\epsilon-6(2-\alpha)\epsilon}
	$$
	which is finite for $a<\frac{4s+\alpha-2}{2}$.
	
	\medskip
	\textbf{Region 3:} $|\xi|\ll|\xi_1-\xi_2|$.
	\begin{equation}
	\sup_{\xi}\int_{H_3}\; \dfrac{\langle\xi\rangle^{2s+2a}{(|\xi_1-\xi_2|+|\xi|)^{(2-\alpha)(1-6\epsilon)}}}{\langle\xi_1\rangle^{2s}\langle\xi_2\rangle^{2s}\langle\xi-\xi_1+\xi_2\rangle^{2s}\left\lbrace\langle\xi_1-\xi\rangle\langle\xi_1-\xi_2\rangle\right\rbrace^{1-6\epsilon} }\;d\xi_1 d\xi_2,\label{eq14lemma5}
	\end{equation}
	where $H_3:=H\cap \left\lbrace|\xi|\ll|\xi_1-\xi_2|\right\rbrace$. Then \eqref{eq14lemma5} is equivalent to
	\begin{equation}
	\sup_{\xi}\int_{H_3}\; \dfrac{\langle\xi\rangle^{2s+2a}\langle \xi_1-\xi_2\rangle^{(2-\alpha)(1-6\epsilon)}}{\langle\xi_1\rangle^{2s}\langle\xi_2\rangle^{2s}\langle\xi-\xi_1+\xi_2\rangle^{2s}\langle\xi_1-\xi\rangle^{1-6\epsilon}\langle\xi_1-\xi_2\rangle^{1-6\epsilon}}\;d\xi_1 d\xi_2,\label{eq15lemma5}
	\end{equation}
	
	On the other hand, since $|\xi|\ll|\xi_1-\xi_2|$, we have $|\xi|\lesssim|\xi-\xi_1+\xi_2|$. Thus \eqref{eq15lemma5} is bounded by
	\begin{equation}
	\begin{aligned}
	\sup_{\xi}\int_{H_3}\; \dfrac{\langle\xi\rangle^{2s+2a}\langle \xi_1-\xi_2\rangle^{(1-\alpha)(1-6\epsilon)-2s}}{\langle\xi_1\rangle^{2s}\langle\xi_2\rangle^{2s}\langle\xi_1-\xi\rangle^{1-6\epsilon}}&\;d\xi_1 d\xi_2
	\\[7pt]
	=\sup_{\xi}\int_{H_3}\; &\dfrac{\langle\xi\rangle^{2s+2a}\langle \xi_1-\xi_2\rangle^{2-\alpha-4s+6\alpha\epsilon}}{\langle\xi_1\rangle^{2s}\langle\xi_2\rangle^{2s}\langle\xi_1-\xi_2\rangle^{1-2s+6\epsilon}\langle\xi_1-\xi\rangle^{1-6\epsilon}}\;d\xi_1 d\xi_2,
	\end{aligned}
	\nonumber
	\end{equation}
	
	since $\frac{2-\alpha}{4}<s$ and $|\xi|\ll|\xi_1-\xi_2|$ we bounded the above by
	\begin{equation}
	\sup_{\xi}\int_{H_3}\; \dfrac{\langle\xi\rangle^{2a+2-\alpha-2s+6\alpha\epsilon}}{\langle\xi_1\rangle^{2s}\langle\xi_2\rangle^{2s}\langle\xi_1-\xi_2\rangle^{1-2s+6\epsilon}\langle\xi_1-\xi\rangle^{1-6\epsilon}}\;d\xi_1 d\xi_2,\label{eq16lemma5}
	\end{equation}
	Finally, we use  Lemma \ref{Lemma1} in the variable $\xi_2$ and taking into count that $s<\frac{1}{2}$, we bounded  \eqref{eq16lemma5} by
	\begin{equation}
	\sup_{\xi}\int_{\xi_1}\; \dfrac{\langle\xi\rangle^{2a+2-\alpha-2s+6\alpha\epsilon}}{\langle\xi_1\rangle^{2s+6\epsilon}\langle\xi_1-\xi\rangle^{1-6\epsilon}}\;d\xi_1\lesssim
	\sup_{\xi}\;\langle\xi\rangle^{2a+2-\alpha-4s+6\alpha\epsilon}
	\end{equation}
	which is finite for $a<\frac{4s+\alpha-2}{2}$.


\end{document}